\documentclass[smallcondensed]{svjour3}       
\smartqed  

\usepackage{latexsym,enumerate,verbatim,cite,lscape,algorithm,algorithmic}
\usepackage{amssymb,amsfonts,amsmath}
\usepackage{microtype}
\usepackage{color} 
\usepackage[active]{srcltx}
\usepackage{rotating}
\usepackage[framemethod=tikz]{mdframed}
\mdfsetup{%
	skipbelow=4pt,
	skipabove=4pt,
	linewidth=1.25pt,
	backgroundcolor=green!11,
	userdefinedwidth=\textwidth,
	roundcorner=10pt,
}
\makeatletter
\newcommand\footnoteref[1]{\protected@xdef\@thefnmark{\ref{#1}}\@footnotemark}
\makeatother%
\def\R{\mathbb{R}}
\def\Argmin{\mathop{\rm Argmin}}

\begin{document}

\title{A difference-of-convex approach for split feasibility with applications to matrix factorizations and outlier detection}


\author{Chen Chen
\and Ting Kei Pong
\thanks{Ting Kei Pong's research was supported partly by Hong Kong Research Grants Council PolyU153085/16p.}
\and Lulin Tan
\thanks{Lulin Tan's research was supported partly by Natural Science Foundation of China (No. 11601162) and Natural Science Foundation of Guangdong Province, China (No. 2017A030310167).}
\and Liaoyuan Zeng
}
\institute{
Chen Chen \at School of Mathematical Sciences, South China Normal University, Guangzhou 510631, China\\
\email{chenchen@m.scnu.edu.cn}
\and
Ting Kei Pong \at
Department of Applied Mathematics, The Hong Kong Polytechnic University, Hong Kong, People's Republic of China\\
\email{tk.pong@polyu.edu.hk}
\and
Lulin Tan \at
School of Mathematical Sciences, South China Normal University, Guangzhou 510631, China\\
\email{lulin\_9@hotmail.com}
\and
Liaoyuan Zeng \at
Department of Applied Mathematics, The Hong Kong Polytechnic University, Hong Kong, People's Republic of China\\
\email{lyzeng@polyu.edu.hk}
}

\date{November 4, 2020}

\maketitle

\begin{abstract}
  The split feasibility problem is to find an element in the intersection of a closed set $C$ and the linear preimage of another closed set $D$, assuming the projections onto $C$ and $D$ are easy to compute. This class of problems arises naturally in many contemporary applications such as compressed sensing. While the sets $C$ and $D$ are typically assumed to be convex in the literature, in this paper, we allow both sets to be possibly nonconvex. We observe that, in this setting, the split feasibility problem can be formulated as an optimization problem with a difference-of-convex objective so that standard majorization-minimization type algorithms can be applied. Here we focus on the nonmonotone proximal gradient algorithm with majorization studied in \cite[Appendix~A]{LiuPongTake2017}. We show that, when this algorithm is applied to a split feasibility problem, the sequence generated clusters at a stationary point of the problem under mild assumptions. We also study local convergence property of the sequence under suitable assumptions on the closed sets involved. Finally, we perform numerical experiments to illustrate the efficiency of our approach on solving split feasibility problems that arise in completely positive matrix factorization, (uniformly) sparse matrix factorization, and outlier detection.
\end{abstract}

\section{Introduction}

The split feasibility problem aims at finding an element common to a closed set $C$ and the linear preimage of another closed set $D$, under the assumption that the projections onto $C$ and $D$ can be computed efficiently. This latter assumption is satisfied by a large class of closed convex sets (to which the projection is unique) including some simple polyhedral sets, and many widely used nonconvex sets such as the set of $s$-sparse vectors (see, for example, \cite[Proposition~3.1]{LuZhan2013}), the simplex with additional cardinality constraints \cite{KyBeCeKo13}, the set of orthogonal matrices \cite[Proposition 7]{AbsilMalick2012} and the set of matrices of rank at most $r$ \cite{EckYoung36}, etc. The split feasibility problem was first introduced in \cite{CensElfv1994}, and has found various applications, such as compressed sensing, signal processing,
image reconstruction and intensity modulated therapy; see, for example, \cite{Byrne2002,CensElfvKopfBort2005,LoMaWaXu12,XuChiYangLang2018} and references therein.

Although the split feasibility problem can be seen as a special case of the classical feasibility problem that finds a point in the intersection of two closed sets, a direct application of algorithms for feasibility problems such as the alternating projection method and the Douglas-Rachford splitting method may not be desirable. This is because, in a split feasibility problem, we only assume that the projections onto $C$ and $D$ are easy to compute; in particular, it can be difficult to project onto the linear preimage of $D$, rendering a direct application of classical methods for feasibility problems inefficient.
Specialized algorithms have thus been proposed for solving split feasibility problems, using only projections onto $C$ and $D$ as well as applications of the linear map and its adjoint.
However, most existing work on split feasibility problems focuses on the convex settings, i.e., they assume $C$ and $D$ are also convex; see, for example \cite{QuXiu2005, Byrne2002, ZhaoYang2005, CensElfvKopfBort2005,CensMotoKopfSega2007,CensElfv1994,LoMaWaXu12,ShehIyio2017,WangYangYang2011,XuChiYangLang2018,Yang2004}. This does not cover contemporary applications that involve nonconvex constraints.

In this paper, we consider the split feasibility problem in a possibly nonconvex setting, i.e., we allow the sets $C$ and $D$ to be possibly nonconvex. We propose an algorithm for solving it and analyze its global and local convergence properties. The algorithm we propose can be viewed as a generalization of the classical CQ algorithm, which was proposed in \cite{Byrne2002} for convex split feasibility problems. Indeed, as we will discuss in more detail in Section~\ref{sec3}, the split feasibility problem can be reformulated into a special possibly nonconvex optimization problem with a difference-of-convex objective function so that standard majorization-minimization type algorithms can be employed. Our main algorithm is an adaptation of the majorization-minimization type algorithm proposed in \cite[Appendix~A]{LiuPongTake2017} to solve this special optimization problem. When $C$ and $D$ are both convex and a constant stepsize strategy is adopted, our proposed algorithm reduces to the classical CQ algorithm.

Since we are solving the split feasibility problem via solving a nonconvex optimization problem, one cannot expect to obtain a global minimizer in general. Instead, we define a new concept of stationary point for split feasibility problem in Definition~\ref{def:stat} below and show that, under mild assumptions, any cluster point of the sequence generated by our algorithm is a {\em stationary point} of the split feasibility problem. The whole sequence generated is further shown to be convergent under additional assumptions such as the Kurdyka-{\L}ojasiewicz property \cite{AttoBolt2009,AttoBoltRedoSoub2010,AttoBoltSvai2013} and Lipschitz differentiability at the limit point. Furthermore, we also analyze local convergence rate, based on the Kurdyka-{\L}ojasiewicz exponent \cite{AttoBolt2009,AttoBoltRedoSoub2010,LiPong2018} and a generalization of the concept of linearly regular intersection: the concept of linearly regular intersection was proposed in \cite{LewiLukeMali2009} for studying local convergence rate of the alternating projection method for feasibility problems.

Finally, we perform numerical experiments to illustrate the efficiency of our method for solving split feasibility problems. Specifically, we perform numerical experiments on the completely positive matrix factorization problem, the (uniformly) sparse matrix factorization problem and an outlier detection problem. In particular, for the completely positive matrix factorization problem, we follow the approach in \cite{GroeDur2018} to reformulate the completely positive matrix factorization problem into a nonconvex split feasibility problem. Our numerical results show that our method always outperforms \cite[Algorithm 2]{GroeDur2018} in terms of both CPU time and solution quality.

The rest of the paper is organized as follows. In Section~\ref{sec2}, we introduce notation and some
preliminary results. The mathematical formulation of the split feasibility problem and our main algorithm for solving it are described in Section~\ref{sec3}. In Section~\ref{sec4}, we study subsequential convergence of the sequence generated by our algorithm. The global (sequential) convergence and the local
convergence rate of the sequence generated by our algorithm are studied in Section~\ref{sec5}. Finally, in Sections~\ref{sec 6.1}, \ref{sec7} and \ref{sec8}, we discuss how our algorithm can be applied
to solving the completely positive matrix factorization problem, the (uniformly) sparse matrix factorization problem and an outlier detection problem, respectively, and perform numerical experiments to study the performance of our algorithm.

\section{Notation and preliminaries}\label{sec2}

In this paper, we let ${\mathbb R}^n$ denote the $n$-dimensional Euclidean space. For a vector $x\in \R^n$, we denote its Euclidean norm, $\ell_1$ norm and $\ell_\infty$ norm by $\|x\|$, $\|x\|_1$ and $\|x\|_\infty$ respectively. We also let $B(x,r)$ denote the closed ball centered at $x$ with radius $r$, i.e., $B(x,r) = \{u\in \R^n:\; \|u - x\|\le r\}$.

An extended real valued function $f:\R^n\to (-\infty,\infty]$ is said to be proper if ${\rm dom}f:=\{x\in \R^n:\; f(x) <\infty\}\neq\emptyset$. Such a function is said to be closed if it is lower semicontinuous.
For a proper closed function $f$, the regular subdifferential and the (limiting) subdifferential of $f$ at an $\bar x\in {\rm dom}f$ are respectively defined as \cite[Definition 8.3]{RockWets1998}
\[
{\hat\partial} f(\bar x):=\left\{v\in {\mathbb R}^n:\; \liminf_{x\to\bar x,x\neq \bar x}\frac{f(x)-f(\bar x)-\langle v,x-\bar x\rangle}{\Vert x-\bar x\Vert}\geq 0\right\},
\]
 and
\[
\partial f(\bar x):=\{v\in \R^n:\;\exists x^t\stackrel{f}{\longrightarrow}\bar x~{\rm and }~v^t\to v~{\rm with}~v^t\in {\hat\partial} f(x^t)~\mbox{for each }t\},
\]
where $x^t\stackrel{f}{\longrightarrow}\bar x$ means both $f(x^t)\to f(\bar x)$ and $x^t\to\bar x$. By convention, we also set ${\hat\partial}f(x)=\partial f(x)=\emptyset$ if $ x\notin {\rm dom}f$. We let ${\rm dom}\partial f$ denote the domain of subdifferential, which
is defined as ${\rm dom}\partial f:=\{x\in \R^n:~\partial f(x)\neq \emptyset\}$. It is known in \cite[Exercise~8.8]{RockWets1998} that if $f$ is continuously differentiable at $x\in \R^n$, then $\partial f(x) = \{\nabla f(x)\}$.
In addition, if $f$ is proper convex, then $\partial f$ coincides with the notion of subdifferential in convex analysis; see \cite[Proposition~8.12]{RockWets1998}.

For a nonempty closed set $C\subseteq {\mathbb R}^n$, we let $C^{\infty}$ denote the horizon cone of $C$, which is defined in \cite[Definition 3.3]{RockWets1998} as
\[
C^{\infty}:=\{x\in \R^n:~\exists~x^t\in C~{\rm and}~\beta_t\downarrow 0~{\rm with}~\beta_tx^t\to x\}.
\]
It is known that $C$ is compact if and only if $C^\infty=\{0\}$; see \cite[Theorem~3.5]{RockWets1998}.
We also let $\delta_C$ denote the indicator function of $C$, which is zero in $C$ and equals infinity otherwise.
The distance from $x$ to $C$ is denoted by $d(x,C): = \inf_{u\in C}\|x - u\|$, and we use ${\rm Proj}_C(x)$ to denote the set of projections of $x$ onto $C$,
which is defined as
\[
{\rm Proj}_C(x):=\Argmin_{y\in C}\Vert y-x\Vert,
\]
where $\Argmin$ denotes the set of minimizers. The set of projections onto the nonempty closed set $C$ is always nonempty, and reduces to a singleton set if $C$ is in addition convex.
The regular normal cone and the (limiting) normal cone of a nonempty closed set $C$ at an $x\in C$ are defined by $\hat N_C(x):=\hat \partial \delta_C(x)$ and $N_C(x):=\partial \delta_C(x)$ respectively. These notions of normal cones are closely related to projections. Indeed, if $x\in {\rm Proj}_C(y)$, then we have from \cite[Example 6.16]{RockWets1998} and \cite[Proposition~6.5]{RockWets1998} that
 \begin{equation}\label{normalcone}
 y-x\in \hat N_C(x)\subseteq N_C(x).
 \end{equation}
Finally, following \cite[Definition~6.4]{RockWets1998} and \cite[Definition~7.25]{RockWets1998}, we say that a nonempty closed set $C$ is (Clarke) regular at an $x\in C$ if $N_C(x) = \hat N_C(x)$, and a proper closed function $f$ is regular at an $x\in {\rm dom}\,f$ if its epigraph ${\rm epi}\,f:=\{(x,t)\in \R^n\times \R:\; f(x)\le t\}$ is regular at $(x,f(x))$.

We next recall the Kurdyka-{\L}ojasiewicz (KL) property \cite{AttoBoltRedoSoub2010}. This property has been used extensively in recent years for analyzing the rate of convergence of various first-order methods, especially in a nonconvex setting; see, for example, \cite{AttoBolt2009,AttoBoltRedoSoub2010,AttoBoltSvai2013}.

\begin{definition}[{{\bf KL property}}]\label{def:KL}
We say that a proper closed function $f$ satisfies the KL property at $\bar x\in {\rm dom}\partial f$ if there exist a neighborhood U of $\bar x$, $s\in(0,\infty]$ and a continuous concave function $\psi:[0,s)\to \mathbb{R}_+$ with $\psi(0)=0$ such that:
\begin{enumerate}[{\rm (i)}]
  \item $\psi$ is continuously differentiable on $(0,s)$ with $\psi'>0$;
  \item for all $x\in U$ with $f(\bar x)<f(x)<f(\bar x)+s$, one has
\[
\psi'(f(x)-f(\bar x))d(0,\partial f(x))\geq 1.
\]
\end{enumerate}
A proper closed function $f$ satisfying the KL property at all points in ${\rm dom}\partial f$ is called a KL
function.
\end{definition}

Functions satisfying the KL property arise naturally in many applications. In particular, it is known that any proper closed semialgebraic function is a KL function; see \cite{AttoBoltRedoSoub2010,BoltDaniLewi2007} for more examples. Moreover, for proper closed semialgebraic functions, the $\psi$ in Definition~\ref{def:KL} can actually be chosen as $\psi(a)=ca^{1-\theta}$ for some $c > 0$ and $\theta\in[0,1)$; see \cite[Section~4.3]{AttoBoltRedoSoub2010} and references therein. This exponent $\theta$ is important in estimating the rate of convergence of sequences generated by various first-order methods; see, for example, \cite{AttoBolt2009,AttoBoltRedoSoub2010,LiPong2018}.

\begin{definition}[{{\bf KL exponent}}]
Let $\theta\in [0,1)$ and $f$ be a proper closed function. We say that $f$ satisfies the KL property at $\bar x\in {\rm dom}\partial f$ with exponent $\theta$ if there exist $c, \epsilon>0$ and $s\in(0,+\infty]$ such that
\begin{equation*}
d(0,\partial f(x))\geq c(f(x)-f(\bar x))^{\theta}
\end{equation*}
whenever $\Vert x-\bar x\Vert\leq \epsilon$ and $0<f(x)-f(\bar x)<s$. If $f$ satisfies the KL property with exponent $\theta$ at every $x\in {\rm dom}\partial f$, then we say that $f$ is a KL function with exponent $\theta$.
\end{definition}

\section{Problem statement and difference-of-convex reformulation}\label{sec3}

In this section, we give the formal mathematical definition of the split feasibility problem and describe the basic ideas leading to our solution strategy. Precisely, the split feasibility problem \cite{CensElfv1994} is stated as follows:
Given a matrix $A\in {\mathbb R}^{m\times n}$, and two nonempty closed sets
$C \subseteq {\mathbb R}^n$ and $D \subseteq {\mathbb R}^m$,
\begin{equation}\label{eq 1.1}
{\rm Find}~x\in \R^n~{\rm s.t.}~Ax\in D~{\rm and}~x\in C;
\end{equation}
here, we assume that an element of ${\rm Proj}_C(x)$ and ${\rm Proj}_D(x)$ can be computed efficiently for any given $x$.
The above problem arises in various contemporary applications. For instance, the {\em noiseless} compressed sensing problem was modeled as an instance of \eqref{eq 1.1} in \cite[Section~6]{LoMaWaXu12}, where $D$ is the singleton set containing the noiseless measurement and $C$ is the $\ell_1$ norm ball of suitable radius; notice that ${\rm Proj}_C(x)$ and ${\rm Proj}_D(x)$ can be computed efficiently for this choice of $C$ and $D$.

We would like to point out that \eqref{eq 1.1} can also be viewed as a special case of the classical feasibility problem, where one attempts to find a point in the intersection of two closed sets: in this case, $A^{-1}D$ and $C$. However, classical algorithms for such a feasibility problem typically involve ${\rm Proj}_{A^{-1}D}(x)$, which can be hard to compute even though ${\rm Proj}_D(x)$ can be computed efficiently. Thus, specialized algorithms have been designed for solving \eqref{eq 1.1}, making use of only projections onto $C$ and $D$ as well as multiplications by the matrix $A$ and its transpose. Almost all algorithms proposed for solving \eqref{eq 1.1} were for the {\em convex setting}, i.e., when $C$ and $D$ are both in addition convex. One classical algorithm is the so-called CQ algorithm proposed in \cite[Algorithm~1.1]{Byrne2002}, which takes the following form: given $x^0\in \R^n$ and $\gamma\in (0,\frac{2}{\lambda_{\max}(A^TA)})$, update
\begin{equation}\label{CQ_algorithm}
x^{t+1}={\rm Proj}_C\left(x^t-\gamma A^T[Ax^t - {\rm Proj}_D(Ax^t)]\right).
\end{equation}
After the proposal of the CQ algorithm, many other algorithms for solving the split feasibility problem \eqref{eq 1.1} in the convex setting have been proposed; we refer the interested readers to \cite{Yang2004,QuXiu2005,ZhaoYang2005} for more detail.

In this paper, we consider the split feasibility problem \eqref{eq 1.1} in a possibly nonconvex setting, i.e., we allow the sets $C$ and $D$ to be possibly nonconvex. Our approach is based on a (standard) reformulation of \eqref{eq 1.1} into the following optimization problem:
\begin{equation}\label{eq 1.2}
\min_{x}F\left(x\right):=\frac{1}{2}d^2\left(Ax,D\right)+\delta_C\left(x\right).
\end{equation}
Indeed, it is easy to see that \eqref{eq 1.1} is solved if and only if \eqref{eq 1.2} has an optimal solution with the optimal value being zero.
Thus, in order to solve \eqref{eq 1.1}, it suffices to solve \eqref{eq 1.2}.

In the case when $C$ and $D$ are both convex, problem \eqref{eq 1.2} is a convex optimization problem and the function $x\mapsto \frac{1}{2}d^2\left(Ax,D\right)$ is smooth with Lipschitz gradient whose modulus is $\lambda_{\max}(A^TA)$. Thus, one can apply first-order methods such as the proximal gradient algorithm and its variant for solving \eqref{eq 1.2} efficiently; in particular, in each iteration of these algorithms, one only needs to compute the projections onto $D$ (for evaluating the gradient of $\frac12 d^2(A\,\cdot,D)$) and $C$ as well as multiplications by $A$ and $A^T$, which can be done efficiently. Notice that the classical CQ algorithm \eqref{CQ_algorithm} is just an application of the standard proximal gradient algorithm to \eqref{eq 1.2} in the convex setting; see, for example, the introduction of \cite{LoMaWaXu12}.

However, in the general case when $C$ and $D$ can be both nonconvex, the squared distance function in problem \eqref{eq 1.2} is nonsmooth in general, and the proximal gradient algorithm cannot be applied.
Fortunately, it is known that the squared distance function can be written as the difference of two convex functions \cite{Asplund1973}: specifically, for any $u\in \mathbb{R}^m$, we have
\[
\frac{1}{2}d^2\left(u,D\right)=\frac{1}{2}\Vert u\Vert^2-\sup_{y\in D}\left\{\left<u,y\right>-\frac{1}{2}\Vert y\Vert^2\right\}.
\]
Now, notice that the function $u\mapsto \sup_{y\in D}\{\left<u,y\right>-\frac{1}{2}\Vert y\Vert^2\}$, as the supremum of affine functions and being finite
valued, is convex continuous. Thus, we can write $F$ in \eqref{eq 1.2} as
\begin{equation}\label{eq 3.7}
F\left(x\right)=\underbrace{\frac{1}{2}\Vert Ax\Vert^2}_{h(x)}+\underbrace{\delta_C\left(x\right)}_{P(x)}-\underbrace{\sup_{y\in D}\left\{\left< Ax,y\right>-\frac{1}{2}\Vert y\Vert^2\right\}}_{g(x)},
\end{equation}
where $h$ is quadratic, $P$ is proper closed and $g$ is convex continuous. Moreover, under a mild additional assumption, we can show in the next proposition that for any $x^0\in C$, the set $\{x:\; F(x)\le F(x^0)\}$ is bounded.

\begin{lemma}\label{lemma 3.1}
Consider the split feasibility problem \eqref{eq 1.1} with $C^{\infty}\cap A^{-1}D^{\infty}=\{0\}$. Then for any $x^0\in C$, the set $\{x:\; F(x)\le F(x^0)\}$ is bounded, where $F$ is defined in \eqref{eq 1.2}.
\end{lemma}
\begin{proof}
Suppose to the contrary that the set $\{x:\; F(x)\le F(x^0)\}$ is unbounded. Then there exists a sequence $\{x^t\}$ such that $F(x^t)\le F(x^0)$ for all $t$ and $\lim_{t\to \infty}\Vert x^t\Vert=+\infty$.
By passing to a subsequence if necessary, we may assume without
loss of generality that $\|x^t\|\neq 0$ for all $t$ and that $\frac{x^t}{\|x^t\|}\to y^*$ for some $y^*$. Then it holds that $\Vert y^*\Vert=1$ and we also have $y^*\in C^{\infty}$ from the definition of the horizon cone. Next, for each $t$, we have
\[
\frac12d^2(Ax^t,D) = F(x^t) \le F(x^0) = \frac12d^2(Ax^0,D).
\]
Hence there exists $\{d^t\}\subseteq D$ such that for each $t$,
\[
\Vert Ax^t-d^t\Vert \leq d(Ax^0,D).
\]
Since $\|x^t\|\neq 0$ for all $t$, dividing both sides of the above inequality by $\|x^t\|$, we have for all $t$ that
\[
\left\Vert \frac{Ax^t}{\Vert x^t\Vert}-\frac{d^t}{\Vert x^t\Vert}\right\Vert\leq \frac{d(Ax^0,D)}{\Vert x^t\Vert}.
\]
Passing to the limit as $t\to \infty$ in the above inequality and noting that $\frac{x^t}{\|x^t\|} \to y^*$, we deduce further that $\frac{d^t}{\Vert x^t\Vert}\to Ay^*$.
Thus, we have $Ay^*\in D^{\infty}$, i.e., $y^*\in A^{-1}D^{\infty}$, according to the definition of the horizon cone. Since it also holds that $y^*\in C^\infty$ and $\|y^*\|=1$, we have arrived at a contradiction. This completes the proof.
\qed\end{proof}

\begin{remark}
The condition $C^{\infty}\cap A^{-1}D^{\infty}=\{0\}$ used in Lemma~\ref{lemma 3.1} holds in the following cases.
\begin{enumerate}[{\rm (i)}]
	\item The set $C$ is compact: in this case, we can deduce from \cite[Proposition~2.1.2]{AusTeb03} that $C^\infty = \{0\}$ and hence $C^{\infty}\cap A^{-1}D^{\infty}=\{0\}$.
	\item The sets $C$, $D$ are closed convex with $ C\cap A^{-1}D$ being nonempty and bounded: in this case, we have
 \[
 C^{\infty}\cap A^{-1}D^{\infty}=C^\infty \cap (A^{-1}D)^\infty = ( C\cap A^{-1}D)^\infty = \{0\},
 \]
 where the first equality follows from \cite[Proposition~2.1.11]{AusTeb03} and the fact that $A^{-1}D\neq\emptyset$, the second equality follows from \cite[Proposition~2.1.9]{AusTeb03} and the fact that $ C\cap A^{-1}D\neq\emptyset$, and the last equality follows from \cite[Proposition~2.1.2]{AusTeb03}.
\end{enumerate}
\end{remark}

Consequently, thanks to Lemma~\ref{lemma 3.1}, under the additional assumption that $C^{\infty}\cap A^{-1}D^{\infty}=\{0\}$, the function $F = h + P - g$ given in \eqref{eq 3.7} satisfies all the assumptions necessary for applying the so-called NPG$_{\rm major}$ in \cite[Appendix~A]{LiuPongTake2017}. In addition, since
\begin{equation}\label{Aginclu}
A^T{\rm Proj}_D(Ax)\subseteq \partial g(x),
\end{equation}
one can choose in {\bf Step 1a)} of the NPG$_{\rm major}$ any
$\eta^t\in {\rm Proj}_D(Ax^t)$ so that $\zeta^t:= A^T\eta^t\in \partial g(x^t)$. Moreover, using the definition of $h$ and $P$ in \eqref{eq 3.7}, the subproblem of NPG$_{\rm major}$ in \cite[Eq~(45)]{LiuPongTake2017} becomes
\begin{equation*}
u\in {\rm Proj}_C\left(x^t-\frac{A^T(Ax^t-\eta^t)}{L_t}\right).
\end{equation*}
Having these in mind, we are now ready to present our algorithm ${\sf SpFeas}_{{\rm DC}_{\sf ls}}$ as Algorithm~\ref{alg1} below for solving \eqref{eq 1.1}, which is basically an application of the NPG$_{\rm major}$ in \cite[Appendix~A]{LiuPongTake2017} to \eqref{eq 1.2}.

\begin{algorithm}[h]
\caption{${\sf SpFeas}_{{\rm DC}_{\sf ls}}$ for \eqref{eq 1.1}}\label{alg1}
\begin{algorithmic}
\STATE {\bf Step 0.} Choose $x^0\in C$, $L_{\max}\geq L_{\min}>0$, $\tau>1$, $c>0$ and an integer $M\geq 0$. Set $t=0$.

{\bf Step 1.} Choose any $L_t^0\in [L_{\min},L_{\max}]$ and set $L_t=L_t^0$.

{\bf 1a)} Pick any $\eta^t\in {\rm Proj}_D(Ax^t)$ and find
\begin{equation*}
u\in {\rm Proj}_C\left(x^t-\frac{A^T(Ax^t-\eta^t)}{L_t}\right).
\end{equation*}

{\bf 1b)} If
\begin{equation}\label{eq 3.6}
d^2(Au,D)\leq \max_{[t-M]_+\leq i\leq t}d^2(Ax^i,D)-c\Vert u-x^t\Vert^2
\end{equation}
is satisfied, go to {\bf Step 2).}

{\bf 1c)} Set $L_t\leftarrow \tau L_t$ and go to {\bf Step 1a).}

{\bf Step 2.} If a termination criterion is not met, set ${\bar L}_t=L_t$, $x^{t+1}=u, t=t+1$. Go to {\bf Step 1.}
\end{algorithmic}
\end{algorithm}

Notice that each iteration of ${\sf SpFeas}_{{\rm DC}_{\sf ls}}$ only involves projections onto $C$ and $D$ as well as multiplications by the matrix $A$ and its transpose, which can be performed efficiently.
Moreover, when $C^{\infty}\cap A^{-1}D^{\infty}=\{0\}$, it can be shown as in \cite[Proposition~1]{LiuPongTake2017} that the linesearch criterion in {\bf Step 1b)} must be satisfied after finitely many inner iterations (independent of $t$), and as in \cite[Proposition~2]{LiuPongTake2017} that successive changes of the sequence $\{x^t\}$ go to zero. In summary, we have the following convergence result as an immediate corollary of \cite[Proposition~1]{LiuPongTake2017} and \cite[Proposition~2]{LiuPongTake2017}.

\begin{lemma}\label{lemma 3.2}
Consider the split feasibility problem \eqref{eq 1.1} with $C^{\infty}\cap A^{-1}D^{\infty}=\{0\}$ and let $\{x^t\}$ and $\{\bar L_t\}$ be generated by ${\sf SpFeas}_{{\rm DC}_{\sf ls}}$. Then it holds that $\sup_{t}\bar L_t < \infty$ and $\lim_{t\to\infty}\Vert x^{t+1}-x^t\Vert=0$.
\end{lemma}

Using Lemma~\ref{lemma 3.2}, it is routine to show that, when $C^{\infty}\cap A^{-1}D^{\infty}=\{0\}$, any accumulation point $\bar x$ of $\{x^t\}$ generated by ${\sf SpFeas}_{{\rm DC}_{\sf ls}}$ satisfies
\begin{equation}\label{inclusion:immed}
0 \in \partial h(\bar x) + \partial P(\bar x) - \partial g(\bar x),
\end{equation}
where $h$, $P$ and $g$ are given in \eqref{eq 3.7}. However, in view of the structure of $g$, it is not trivial to completely characterize the set $\partial g$ so as to relate \eqref{inclusion:immed} to the original split feasibility problem \eqref{eq 1.1}. In the next section, we will look at another characterization of the set of accumulation points of $\{x^t\}$ that is more closely tied with \eqref{eq 1.1}.

\section{Subsequential convergence of ${\sf SpFeas}_{{\rm DC}_{\sf ls}}$}\label{sec4}

In this section, we characterize the set of accumulation points of the sequence $\{x^t\}$ generated by ${\sf SpFeas}_{{\rm DC}_{\sf ls}}$ under the assumption $C^{\infty}\cap A^{-1}D^{\infty}=\{0\}$. We start with the following proposition, which concerns the subdifferential of $F$ in \eqref{eq 1.2}.

\begin{proposition}\label{subF}
Consider the split feasibility problem \eqref{eq 1.1} and let $F$ be defined in \eqref{eq 1.2}. Then for every $x\in C$, we have
\begin{equation}\label{eq:subF}
  \partial F(x) \subseteq A^TA{x}-A^T{\rm Proj}_D(Ax)+N_C(x).
\end{equation}
If in addition $C$ is regular at some $\bar x\in C$ and the function $y\mapsto d_D(y)$ is regular at $A\bar x\in \R^m$, then we have
\begin{equation*}
  \partial F(\bar x) = A^TA{\bar x}-A^T{\rm Proj}_D(A\bar x)+N_C(\bar x).
\end{equation*}
\end{proposition}

\begin{remark}
  In view of \cite[Example~8.53]{RockWets1998} and \cite[Corollary~8.11]{RockWets1998}, we see that the function $y\mapsto d_D(y)$ is regular at $A\bar x$ under one of the following conditions:
   \begin{enumerate}[{\rm (i)}]
     \item $A\bar x\notin D$ and ${\rm Proj}_D(A\bar x)$ is a singleton set;
     \item $A\bar x\in D$ and $D$ is regular at $A\bar x$.
   \end{enumerate}
\end{remark}

\begin{proof}
Notice from \cite[Example~8.53]{RockWets1998} and \cite[Theorem~1.110(ii)]{Mordukhovich2006} that for any $y\in \R^m$, we have
\begin{equation}\label{tocite}
\partial\left(\frac12 d^2(\cdot, D)\right)(y) = d(y,D)\partial(d(\cdot, D))(y) = y - {\rm Proj}_D(y).
\end{equation}
This together with \cite[Corollary 10.9]{RockWets1998} and \cite[Theorem 10.6]{RockWets1998} gives \eqref{eq:subF}.

Now, assume in addition that $C$ is regular at some $\bar x\in C$ and $y\mapsto d_D(y)$ is regular at $A\bar x\in \R^m$. Then at $\bar y := A\bar x$, we have
\begin{equation*}
\begin{aligned}
&\partial \left(\frac12 d^2(\cdot,D)\right)(\bar y)\overset{\rm (a)}=d(\bar y,D)\partial \left(d(\cdot,D)\right)(\bar y)\\
&\overset{\rm (b)}=d(\bar y,D)\hat\partial \left(d(\cdot,D)\right)(\bar y)\overset{{\rm (c)}} {\subseteq} \hat\partial \left(\frac12 d^2(\cdot,D)\right)(\bar y)\subseteq \partial \left(\frac12 d^2(\cdot,D)\right)(\bar y)
\end{aligned}
\end{equation*}
where (a) follows from \cite[Theorem 1.110 (ii)]{Mordukhovich2006}, (b) follows from the assumption that $y\mapsto d_D(y)$ is regular at $\bar y$ and \cite[Corollary~8.11]{RockWets1998} and {\rm (c)} can be verified directly from the definition. Thus,
\begin{equation*}
\partial \left(\frac12 d^2(\cdot,D)\right)(\bar y)=\hat \partial \left(\frac12 d^2(\cdot,D)\right)(\bar y).
\end{equation*}
This together with \cite[Corollary 8.11]{RockWets1998} implies that the function $y\mapsto \frac12 d^2(y,D)$ is regular at $\bar y$. Using this together with
\cite[Theorem 10.6]{RockWets1998}, we deduce further that the function $x\mapsto \frac{1}{2}d^2(Ax,D)$ is regular at $\bar x$ and
\[
\partial \left(\frac{1}{2}d^2(A\,\cdot,D)\right)(\bar x)= A^T\partial \left(\frac{1}{2}d^2(\cdot,D)\right)(A\bar x) = A^T(A\bar x - {\rm Proj}_D(A\bar x)),
\]
where the last equality follows from \eqref{tocite}. Finally, since $C$ is regular at $\bar x$, we have from \cite[Example 7.28]{RockWets1998} that $\delta_C$ is regular at $\bar x$. The desired conclusion now follows from \cite[Corollary 10.9]{RockWets1998}.
\qed\end{proof}

Notice that if $x^*$ solves \eqref{eq 1.1}, then it also solves \eqref{eq 1.2}. According to \cite[Theorem~10.1]{RockWets1998}, we must then have $0\in \partial F(x^*)$. Motivated by this observation and Proposition~\ref{subF}, we make the following definition.
\begin{definition}[{{\bf Stationary points of \eqref{eq 1.1}}}]\label{def:stat}
  For the split feasibility problem \eqref{eq 1.1}, we say that $x^*$ is a stationary point of this problem if
  \[
0\in A^TAx^*-A^T{\rm Proj}_D(Ax^*)+N_C(x^*).
\]
\end{definition}

Based on \cite[Theorem~10.1]{RockWets1998} and Proposition~\ref{subF}, we see that if $x^*$ solves \eqref{eq 1.1} (and hence \eqref{eq 1.2}), then it is a stationary point of \eqref{eq 1.1}. Moreover, if $C$ is regular at $x^*\in C$ and $y\mapsto d_D(y)$ is regular at $Ax^*\in \R^m$, then $x^*$ being stationary for problem \eqref{eq 1.1} is the same as $x^*$ being a stationary point of the function $F$ defined in \eqref{eq 1.2}, in the sense that $0\in \partial F(x^*)$. Finally, in view of \eqref{Aginclu}, we deduce that if $x^*$ is a stationary point of \eqref{eq 1.1} in the sense of Definition~\ref{def:stat}, then it satisfies \eqref{inclusion:immed} in place of $\bar x$. This shows that the notion of stationarity of \eqref{eq 1.1} defined in Definition~\ref{def:stat} is at least as strong as the condition \eqref{inclusion:immed}. 

We next show that the sequence generated by ${\sf SpFeas}_{{\rm DC}_{\sf ls}}$ clusters at a stationary point of \eqref{eq 1.1}.

\begin{theorem}[{{\bf Subsequential convergence of ${\sf SpFeas}_{{\rm DC}_{\sf ls}}$}}]\label{thm 3.1}
Consider the split feasibility problem \eqref{eq 1.1} with $C^{\infty}\cap A^{-1}D^{\infty}=\{0\}$ and
let $\{x^t\}$ be the sequence generated by ${\sf SpFeas}_{{\rm DC}_{\sf ls}}$. Then the following statements hold:
\begin{enumerate}[{\rm (i)}]
\item The sequence $\{x^t\}$ is bounded.
\item Any accumulation point of $\{x^t\}$ is a stationary point of \eqref{eq 1.1}.
 \end{enumerate}
\end{theorem}
\begin{proof}
The boundedness of $\{x^t\}$ follows from Lemma~\ref{lemma 3.1} and \eqref{eq 3.6}.

Next, let $x^*$ be an accumulation point of $\{x^t\}$, which exists because the sequence is bounded. Then there exists a convergent subsequence $\{x^{t_j}\}$ such
that $\lim_{j\to\infty}x^{t_j}=x^*$. Clearly, $x^*\in C$ because $C$ is closed. Now, using \eqref{normalcone} and the definition of $x^{t+1}$ as a projection of $x^t-{\bar L_t}^{-1}[A^T(Ax^t-\eta^t)]$ onto $C$, we have for each $j$ that
\begin{equation}\label{eq 3.2}
0\in{\bar L}_{t_j}(x^{t_j+1}-x^{t_j})+A^TAx^{t_j}-A^T\eta^{t_j}+N_C(x^{t_j+1});
\end{equation}
moreover, $\{\bar L_{t_j}\}$ is bounded thanks to Lemma~\ref{lemma 3.2}.

On the other hand, notice from the definition of $\eta^{t_j}$ as projection that $F(x^{t_j})=\frac{1}{2}\Vert Ax^{t_j}-\eta^{t_j}\Vert^2$ for all $j$, where $F$ is defined in \eqref{eq 1.2}. Since $F(x^{t_j})\leq F(x^0)$ for all $j$ according to \eqref{eq 3.6} and the sequence $\{x^t\}$ is bounded according to (i), we deduce further that $\{\eta^{t_j}\}$ is bounded. By passing to a further subsequence if necessary, we may assume without loss of generality that $\eta^{t_j}\to\eta^*$ for some $\eta^*$. Since $\eta^{t_j}\in{\rm Proj}_D(Ax^{t_j})$ and $D$ is closed, we have $\eta^*\in D$. Also, we have
\[
\frac12 d^2(Ax^*,D)=\lim_{j\to\infty}\frac12 d^2(Ax^{t_j},D)=\lim_{j\to\infty}\frac{1}{2}\Vert Ax^{t_j}-\eta^{t_j}\Vert^2=\frac{1}{2}\Vert Ax^*-\eta^*\Vert^2.
\]
This shows that $\eta^*\in{\rm Proj}_D(Ax^*)$. Now, passing to the limit as $j\to \infty$ in \eqref{eq 3.2} and invoking Lemma \ref{lemma 3.2} and
the closedness of $x\rightrightarrows N_C(x)$ at $x^*\in C$, we obtain
\[
0\in A^TAx^*-A^T\eta^*+N_C(x^*),
\]
showing that $x^*$ is a stationary point of \eqref{eq 1.1}. This completes the proof.
\qed\end{proof}

The algorithm ${\sf SpFeas}_{{\rm DC}_{\sf ls}}$ involves a linesearch subroutine in each iteration. We next discuss a variant that does not require any linesearch procedure, which is presented in Algorithm~\ref{alg2} below as ${\sf SpFeas}_{\rm DC}$, where
\begin{equation}\label{rC}
r_C := \begin{cases}
  2 & \mbox{if $C$ is convex},\\
  1 & \mbox{otherwise}.
\end{cases}
\end{equation}
We will argue that ${\sf SpFeas}_{\rm DC}$ is a special instance of ${\sf SpFeas}_{{\rm DC}_{\sf ls}}$.

\begin{algorithm}[htb]
\caption{${\sf SpFeas}_{\rm DC}$ for \eqref{eq 1.1}.}\label{alg2}
\begin{algorithmic}
\STATE {\bf Step 0.} Choose $x^0\in C$ and $L > \frac{\lambda_{\max}(A^TA)}{r_C}$.

{\bf Step 1.} For each $t=0,1,2,...$, pick any $\eta^t\in {\rm Proj}_D(Ax^t)$ and set
\[
x^{t+1} \in {\rm Proj}_C\left(x^t-\frac{A^T(Ax^t-\eta^t)}{L}\right).
\]

{\bf Step 2.} If a termination criterion is not met, go to {\bf Step 1}.
\end{algorithmic}
\end{algorithm}

To this end, we first prove the following proposition.
\begin{proposition}\label{pro 3.1}
Consider the split feasibility problem \eqref{eq 1.1} and let $F$ be defined as in \eqref{eq 1.2}.
Let $L>0$, $x\in C$, $\eta\in {\rm Proj}_D(Ax)$ and set
\begin{equation}\label{eq 3.3}
u\in {\rm Proj}_C\left(x-\frac{A^T(Ax-\eta)}{L}\right).
\end{equation}
Then
\begin{equation}\label{Fdescent}
F(u)\leq F(x)-\frac{r_C L-\lambda_{\max}(A^TA)}{2}\Vert u-x\Vert^2,
\end{equation}
where $r_C$ is defined as in \eqref{rC}.
\end{proposition}
\begin{proof}
Using the decomposition of $F$ in \eqref{eq 3.7}, we have for any $u\in C$ that
\begin{equation}\label{eq 3.4}
  \begin{aligned}
    F(u) & = h(u) + P(u) - g(u) = \frac12\|Au\|^2 + \delta_C(u) - g(u)\\
    & \le \frac12\|Ax\|^2 + (Ax)^T(Au - Ax) + \frac{\lambda_{\max}(A^TA)}{2}\|u - x\|^2 - \eta^T(Au-Ax) - g(x)\\
    & = F(x) + (Ax - \eta)^T(Au - Ax) + \frac{\lambda_{\max}(A^TA)}{2}\|u - x\|^2,
  \end{aligned}
\end{equation}
where the inequality holds because of the Taylor's inequality applied to $w\mapsto \frac12\|Aw\|^2$ and the fact that $A^T\eta \in \partial g(x)$ (see \eqref{Aginclu}).

Now, suppose that $C$ is not convex. Then we see from \eqref{eq 3.3} that
\[
\left\| u - x + \frac{A^T(Ax-\eta)}{L}\right\|^2\le \left\|\frac{A^T(Ax-\eta)}{L}\right\|^2,
\]
because $x\in C$. Thus, $\frac{L}2\|u-x\|^2 + (u-x)^TA^T(Ax-\eta)\le 0$. This together with \eqref{eq 3.4} shows that \eqref{Fdescent} holds (with $r_C=1$) when $C$ is nonconvex.

On the other hand, if $C$ is convex, then the function $\rho(y):=(A^TAx-A^T\eta)^T(y-x)+\frac{L}{2}\Vert y-x\Vert^2+\delta_C(y)$ is a strongly
convex function with modulus $L$. Moreover, we see from \eqref{eq 3.3} that $u$ is the unique minimizer of $\rho$. Thus, we have $\rho(x)-\rho(u)\geq \frac{L}2\Vert x - u\Vert^2$, which gives
\[
L\|u-x\|^2 + (u-x)^TA^T(Ax-\eta)\le 0.
\]
This together with \eqref{eq 3.4} shows that \eqref{Fdescent} holds (with $r_C=2$) when $C$ is convex.
This completes the proof.
\qed\end{proof}

We can now argue that ${\sf SpFeas}_{\rm DC}$ is a special instance of ${\sf SpFeas}_{{\rm DC}_{\sf ls}}$. To this end, pick an $x^0\in C$ and suppose that an $L > \frac{\lambda_{\max}(A^TA)}{r_C}$ is chosen. Let $c>0$ be such that
\begin{equation}\label{c_choice}
L = \frac{c + \lambda_{\max}(A^TA)}{r_C} > 0.
\end{equation}
If we use this $c$ in ${\sf SpFeas}_{{\rm DC}_{\sf ls}}$,
set $L_{\max} = L_{\min} = L$ and fix any nonnegative integer $M$ and any $\tau > 1$, then, according to Proposition~\ref{pro 3.1}, the linesearch condition in \eqref{eq 3.6} is always satisfied with $L_t = L^0_t = L$. Hence, ${\sf SpFeas}_{\rm DC}$ initialized at $x^0\in C$ with the chosen $L > \frac{\lambda_{\max}(A^TA)}{r_C}$ generates the same sequence as ${\sf SpFeas}_{{\rm DC}_{\sf ls}}$ initialized at $x^0\in C$ with $c$ chosen as in \eqref{c_choice} and $L_{\max} = L_{\min} = L$. We thus have the following immediate corollary concerning the convergence of ${\sf SpFeas}_{\rm DC}$.
\begin{corollary}[{{\bf Subsequential convergence of ${\sf SpFeas}_{\rm DC}$}}]\label{coro 4.1}
Consider the split feasibility problem \eqref{eq 1.1} with $C^{\infty}\cap A^{-1}D^{\infty}=\{0\}$ and let $\{x^t\}$ be the sequence generated by ${\sf SpFeas}_{\rm DC}$. Then the following statements hold:
\begin{enumerate}[{\rm (i)}]
\item The sequence $\{x^t\}$ is bounded.
\item Any accumulation point of $\{x^t\}$ is a stationary point of \eqref{eq 1.1}.
\end{enumerate}
\end{corollary}

On passing, we note that ${\sf SpFeas}_{{\rm DC}}$ reduces to the classical CQ algorithm \eqref{CQ_algorithm} when $C$ and $D$ are both convex.\footnote{This is because in this case, we have $r_C=2$ and hence ${\sf SpFeas}_{{\rm DC}}$ corresponds to \eqref{CQ_algorithm} with $\gamma = \frac1L$.} Thus, our algorithm ${\sf SpFeas}_{{\rm DC}_{\sf ls}}$ is a generalization of the classical CQ algorithm for solving \eqref{eq 1.1} in the general nonconvex setting.

\section{Sequential convergence based on KL property}\label{sec5}

\subsection{Global convergence}
In this section, we establish the convergence of the whole sequence generated by ${\sf SpFeas}_{{\rm DC}_{\sf ls}}$ with $M = 0$ under the KL property and some mild assumptions. The KL property has been used extensively in recent years for establishing global convergence of the sequence generated by various first-order methods; see, for example, \cite{AttoBolt2009,AttoBoltRedoSoub2010,AttoBoltSvai2013}. Our proof for the next theorem follows closely the arguments in \cite{AttoBoltSvai2013} and is routine. We include its proof in the appendix for the ease of readers.

\begin{theorem}[{{\bf Global convergence of ${\sf SpFeas}_{{\rm DC}_{\sf ls}}$ with $M=0$}}]\label{thm 4.1}
Consider the split feasibility problem \eqref{eq 1.1} with $C^{\infty}\cap A^{-1}D^{\infty}=\{0\}$ and let the function $F$ in \eqref{eq 1.2} be a KL function. Let $\{x^t\}$ be the sequence generated by ${\sf SpFeas}_{{\rm DC}_{\sf ls}}$ with $M=0$. Then the sequence $\{x^t\}$ is bounded. Let $x^*$ be an accumulation point of $\{x^t\}$ and suppose that the function $y\mapsto \frac{1}{2}d^2(y,D)$ is continuously differentiable at $Ax^*$ with locally Lipschitz gradient. Then the whole sequence $\{x^t\}$ is convergent.
\end{theorem}

Theorem~\ref{thm 4.1} requires the function $F$ in \eqref{eq 1.2} to have the KL property. This particular condition is not very restrictive. Indeed, according to \cite[Section~4.3]{AttoBoltRedoSoub2010}, the function $F$ satisfies the KL property when $C$ and $D$ are both in addition semialgebraic. On the other hand, since $D$ is possibly nonconvex, the function $y\mapsto \frac{1}{2}d^2(y,D)$ may not be differentiable in general at $Ax^*$, where $x^*$ is an accumulation point of $\{x^t\}$, as required by Theorem~\ref{thm 4.1}. In the next two propositions, we give simple sufficient conditions for $y\mapsto \frac{1}{2}d^2(y,D)$ to be continuously differentiable at $Ax^*$ with
locally Lipschitz gradient. The first proposition concerns prox-regularity and is an immediate consequence of \cite[Proposition 8.1]{LewiLukeMali2009}. Recall from \cite[Theorem 1.3]{PoliRockThib2000} that a closed set $\Omega$ is said to be prox-regular at an $\bar x\in \Omega$ if there exists $\epsilon>0$ such that ${\rm Proj}_\Omega(x)$ is single-valued for every $x\in B(\bar x,\epsilon)$.

\begin{proposition}
Consider the split feasibility problem \eqref{eq 1.1}.
If $x^*\in C\cap A^{-1}D$ and $D$ is prox-regular at $Ax^*$, then $y\mapsto \frac{1}{2}d^2(y,D)$ is continuously differentiable at $Ax^*$ with locally Lipschitz gradient.
\end{proposition}
\begin{proof}
The desired result follows directly from \cite[Proposition 8.1]{LewiLukeMali2009}.
\qed\end{proof}

\begin{proposition}
Consider the split feasibility problem \eqref{eq 1.1} and suppose that $D=\bigcup_{i=1}^mD_i$, where each $D_i$, $i=1,\ldots,m$, is closed and convex. If $I(Ax^*):=\{i:d(Ax^*,D)=d(Ax^*,D_i)\}$ is a singleton set, then $y\mapsto \frac{1}{2}d^2(y,D)$ is continuously differentiable at $Ax^*$ with locally Lipschitz gradient.
\end{proposition}
\begin{proof}
It is easy to see that $\frac12d^2(y,D)=\min_{1\le i\le m}\frac12 d^2(y,D_i)$ for all $y\in \R^m$. Moreover, for each $i$, the function $y\mapsto \frac12 d^2(y,D_i)$ is continuously differentiable with Lipschitz gradient because $D_i$ is closed and convex. Write $y^* = Ax^*$ for notational simplicity and note that $I(y^*) = \{i_0\}$ for some $i_0\in \{1,\ldots,m\}$ by assumption.

From the definition of $I(y^*)$ we have $\min_{i\notin I(y^*)}\frac12 d^2(y^*,D_i)>\frac12d^2(y^*,D)$. By continuity, it then holds that for all $y$ sufficiently close to $y^*$, we have
\[
\min_{i\notin I(y^*)}\frac12 d^2(y,D_i)>\frac12d^2(y,D).
\]
Hence $I(y)=\{i_0\}$ for all $y$ sufficiently close to $y^*$. Thus, it holds that $\frac12d^2(y,D) = \frac12d^2(y,D_{i_0})$ locally around $y^*$. Consequently, the function $y\mapsto \frac12d^2(y,D)$ is continuously differentiable at $y^*$ with locally Lipschitz gradient. This completes the proof.
\qed\end{proof}

\subsection{Local convergence behavior}

In this section, we study the local convergence rate of the sequence $\{x^t\}$ generated by ${\sf SpFeas}_{{\rm DC}_{\sf ls}}$ with $M =0$.
Local convergence rates of various first-order methods have been widely studied recently and they are usually analyzed based on the so-called KL exponent of a certain potential function; see, for example, \cite{AttoBolt2009,AttoBoltRedoSoub2010,LiPong2018}.
Here, our analysis uses the $F$ in \eqref{eq 1.2} as the potential function and makes use of the assumption that $F$ is a KL function with exponent $\theta\in [0,1)$.
We first show in the proposition below that this latter assumption holds when $C$ and $D$ are subanalytic sets and $C$ is bounded (see \cite[Definition 2.1 (ii)]{BoltDaniLewi2007} for the definition of subanalytic sets).
\begin{proposition}
  Consider the split feasibility problem \eqref{eq 1.1} and let $F$ be defined in \eqref{eq 1.2}. If $C$ and $D$ are subanalytic sets and $C$ is compact, then $F$ is a KL function with exponent $\theta\in [0,1)$.
\end{proposition}
\begin{proof}
  First of all, since $D$ is subanalytic, we see from p3 and p5 of \cite[page 597]{facchinei2003finite} that the squared distance function $x\mapsto \frac{1}{2}d^2(Ax,D) $ is a subanalytic function, which means that its graph, given by $  \{(x,\frac{1}{2}d^2(Ax,D)):\; x\in\mathbb{R}^n \}$, is a subanalytic set. Next, notice that the graph of $F$ in \eqref{eq 1.2} is given by
\begin{equation*}
\{(x, \textstyle\frac{1}{2}d^2(Ax,D)):\; x\in C \} = \{(x,\frac{1}{2}d^2(Ax,D)):\; x\in\mathbb{R}^n\}\cap (C\times \R).
\end{equation*}
In addition, the set $ C\times \R $ is subanalytic as $ C $ is subanalytic, thanks to p2 of \cite[page 597]{facchinei2003finite}. In view of these, we deduce from p1 of \cite[page 597]{facchinei2003finite} that the graph of $F$ is also subanalytic. In addition, since $C$ is compact, and $x\mapsto \frac{1}{2}d^2(Ax,D)$ is continuous, the graph of $F$ is also compact. Thus, $ F $ is a globally subanalytic function (see \cite[Definition~2.2]{BoltDaniLewi2007} and discussions therein). In view of \cite[Corollary~9]{BolteDaniLewiShiota07} and \cite[Corollary~16]{BolteDaniLewiShiota07}, we conclude that $ F $ is a KL function with exponent $ \theta\in[0,1) $. This completes the proof.
\end{proof}
We now state in the following theorem our local convergence result for ${\sf SpFeas}_{{\rm DC}_{\sf ls}}$ with $M=0$ based on the KL exponent $\theta$ of the $F$ in \eqref{eq 1.2}. The proof is standard and follows a similar line of arguments as in \cite[Theorem~2]{AttoBolt2009}, is thus omitted for brevity.

\begin{theorem}[{{\bf Local convergence rate}}]\label{thm 5.1}
Consider the split feasibility problem \eqref{eq 1.1} with $C^{\infty}\cap A^{-1}D^{\infty}=\{0\}$
and let the function $F$ in \eqref{eq 1.2} be a KL function with exponent $\theta\in [0,1)$.
Let $\{x^t\}$ be the sequence generated by ${\sf SpFeas}_{{\rm DC}_{\sf ls}}$ with $M=0$ and let $x^*$ be its accumulation point. If $y\mapsto \frac{1}{2}d^2(y,D)$ is continuously differentiable at $Ax^*$ with locally Lipschitz gradient, then the following statements hold:
\begin{enumerate}[{\rm (i)}]
  \item If $\theta=0$, then there exists ${\bar t}\geq0$ such that $x^t= x^*$ whenever $t\geq {\bar t}$;
  \item If $\theta\in(0,\frac{1}{2}]$, then there exist ${\bar t}\geq0, {\bar d}>0$ and $\sigma\in (0,1)$ such that $\Vert x^t-x^*\Vert\leq {\bar d}\sigma^t$ for all $t\geq {\bar t}$;
  \item If $\theta\in(\frac{1}{2},1)$, then there exist ${\bar t}\geq0$ and ${\bar d}>0$ such that $\Vert x^t-x^*\Vert\leq {\bar d}t^{-\frac{1-\theta}{2\theta-1}}$ for all $t\geq {\bar t}$.
\end{enumerate}
\end{theorem}

From Theorem \ref{thm 5.1}, we know that if the function $F$ in \eqref{eq 1.2} satisfies the KL property with exponent $\frac{1}{2}$ and a certain differentiability assumption holds at an accumulation point of $\{x^t\}$, then the sequence $\{x^t\}$ generated is locally linearly convergent. We next give sufficient conditions on $C$ and $D$ in \eqref{eq 1.1} so that the $F$ in \eqref{eq 1.2} satisfies the KL property with exponent $\frac{1}{2}$.
Our first result concerns polyhedrality.
\begin{theorem}
Consider the split feasibility problem \eqref{eq 1.1} and let $F$ be defined in \eqref{eq 1.2}. If $C$ and $D$ are both unions of polyhedral sets, then $F$ is a KL function with exponent $\frac{1}{2}$.
\end{theorem}
\begin{proof}
Let $D=\bigcup_{i=1}^kD_i$ and $C=\bigcup_{j=1}^\ell C_j$, where $D_1,\ldots,D_k$ and $C_1,\ldots,C_\ell$ are all polyhedral sets. Then one can see that
\begin{equation}\label{defF0}
F(x)=\min_{\substack{1\le i\le k\\ 1\le j\le \ell}}\underbrace{\frac{1}{2}d^2(Ax,D_i)+\delta_{C_{j}}(x)}_{F_{i,j}(x)}.
\end{equation}
Since $D_i$ and $C_j$ are polyhedral for each $i$ and $j$, we conclude from \cite[Exercise 10.22]{RockWets1998} and \cite[Example 12.31]{RockWets1998} that $F_{i,j}$ is piecewise linear-quadratic for each $i$ and $j$. Using the definition of piecewise linear-quadratic function, we can further rewrite $F_{i,j}$ as follows:
\begin{equation}\label{defF1}
F_{i,j}(x)=\min_{1\leq \nu\leq l_{i,j}}\left\{\frac{1}{2}x^TG_{i,j,\nu}x+\alpha^T_{i,j,\nu}x+\beta_{i,j,\nu}+\delta_{P_{i,j,\nu}}(x)\right\},
\end{equation}
where $G_{i,j,\nu}\in \R^{n\times n}$ is symmetric, $\alpha_{i,j,\nu}\in {\mathbb R}^n$, $\beta_{i,j,\nu}\in \mathbb{R}$, and $P_{i,j,\nu}$ is polyhedral for each $\nu\in \{1,\ldots,l_{ij}\}$. The desired
conclusion now follows from \eqref{defF0}, \eqref{defF1} and \cite[Corollary 5.2]{LiPong2018}. This completes the proof.
\qed\end{proof}

Our next result concerns a certain kind of regularity condition, defined as follows.
\begin{definition}[{{\bf Linearly regular intersection with respect to $A$}}]\label{def:linreg}
  Consider the split feasibility problem \eqref{eq 1.1}. We say that the pair of sets $\{C,A^{-1}D\}$ has a linearly regular intersection with respect to $A$ at a point $x^*\in C\cap A^{-1}D$ if the following implication holds:
  \begin{equation}
  \begin{aligned}\label{eq 5.3}
A^Tw^*+v^*=0&\mbox{ for some }w^*\in N_D(Ax^*)\mbox{ and }v^*\in N_C(x^*),\\
&\Longrightarrow w^*=0 \mbox{ and } v^*=0.
\end{aligned}
\end{equation}
\end{definition}

The concept of linearly regular intersection defined above for split feasibility problems is a generalization of the corresponding property for classical feasibility problems. Recall from \cite[Section 2]{LewiLukeMali2009} that
the pair of nonempty closed sets $\{C_1,C_2\}$ has linearly regular intersection at a point $x^*\in C_1\cap C_2$
if the following implication holds:
\begin{equation*}
\begin{aligned}
v_1+v_2=0&\mbox{ for some }v_1\in N_{C_1}(x^*)\mbox{ and }v_2\in N_{C_2}(x^*),\\
&\Longrightarrow v_1 = v_2 = 0.
\end{aligned}
\end{equation*}
It was proved in \cite[Theorem 5.16]{LewiLukeMali2009} that if $\{C_1,C_2\}$ has linearly regular intersection at an $x^*\in C_1\cap C_2$ and at least one of these two sets is super-regular at $x^*$ (see \cite[Definition~4.3]{LewiLukeMali2009}), then the sequence generated by the alternating projection algorithm for finding a point in $C_1\cap C_2$ is {\em locally linear convergent} as long as the algorithm was initialized sufficiently close to $x^*$. Here, we will show in Theorem~\ref{thm:regimpKL} below that \eqref{eq 5.3} has a similar implication on split feasibility problem \eqref{eq 1.1}: under \eqref{eq 5.3}, the function $F$ defined in \eqref{eq 1.2} has the KL property with exponent $\frac{1}{2}$ at the point $x^*\in C\cap A^{-1}D$.

We start with an auxiliary lemma.
\begin{lemma}\label{lemma 5.1}
Consider the split feasibility problem \eqref{eq 1.1}. Let $x^*\in C\cap A^{-1}D$ and suppose that the pair of sets $\{C,A^{-1}D\}$ has a linearly regular intersection with respect to $A$ at $x^*$, i.e., \eqref{eq 5.3} holds.
Then the following statements hold:
\begin{enumerate}[{\rm (i)}]
\item There exist $\gamma_1>0$ and $\epsilon_1>0$ such that
\begin{equation}\label{eq 5.1}
\Vert A^T(Ax-\eta)\Vert\geq\gamma_1\Vert Ax-\eta\Vert
\end{equation}
whenever $x\in B(x^*,\epsilon_1)$ and $\eta\in {\rm Proj}_D(Ax)$.

\item There exist $\gamma_2\in [0,1)$ and $\epsilon_2>0$ such that
\begin{equation}\label{eq 5.4}
v^TA^T(Ax-\eta)\geq -\gamma_2\Vert A^T(Ax-\eta)\Vert \Vert v\Vert
\end{equation}
whenever $x\in B(x^*,\epsilon_2)\cap C$, $\eta\in {\rm Proj}_D(Ax)$ and $v\in N_C(x)$.
\end{enumerate}
\end{lemma}
\begin{proof}
We first prove (i). Suppose to the contrary that (i) does not hold. Then there exist $\{x^t\}$ and $\{\eta^t\}$ satisfying $x^t\to x^*$, $\eta^t\in {\rm Proj}_D(Ax^t)$
and
\begin{equation}\label{eq 5.2}
\Vert A^T(Ax^t-\eta^t)\Vert<\frac{1}{t}\Vert Ax^t-\eta^t\Vert
\end{equation}
for all $t\geq 1$. In particular, we have $Ax^t-\eta^t\neq 0$ for all $t\ge 1$. Moreover, by passing to a subsequence if necessary, we may assume without loss of generality that
\begin{equation}\label{unitnorm}
\frac{Ax^t - \eta^t}{\|Ax^t - \eta^t\|}\to q^*
\end{equation}
for some $q^*$ with $\|q^*\|=1$.

Now, since $\eta^t\in {\rm Proj}_D(Ax^t)$ and $Ax^*\in D$, we have for all $t\ge 1$ that
\[
\Vert\eta^t-Ax^t\Vert\leq \Vert Ax^t-Ax^*\Vert.
\]
This together with $x^t\to x^*$ implies that $\eta^t\to Ax^*$. In addition, the relation $\eta^t\in {\rm Proj}_D(Ax^t)$ together with \eqref{normalcone} implies that for all $t\ge 1$,
\[
Ax^t - \eta^t \in N_D(\eta^t).
\]
Combining this with \eqref{unitnorm}, the fact that $\eta^t\to Ax^*$ and the closedness of the normal cone mapping yields $q^*\in N_D(Ax^*)$.
Next, divide both sides of \eqref{eq 5.2} by $\Vert Ax^t-\eta^t\Vert$ and pass to the limit, we obtain $A^Tq^*=0$. This together with $q^*\in N_D(Ax^*)$ and \eqref{eq 5.3} gives $q^*=0$, which is a contradiction. This proves (i).

We now prove (ii). Suppose to the contrary that (ii) does not hold. Then there exist $\{x^t\}\subseteq C$, $\{\gamma_2^t\}$, $\{\eta^t\}$ and $\{v^t\}$ satisfying $x^t\to x^*$, $\gamma^t_2\uparrow 1$, $\eta^t\in{\rm Proj}_D(Ax^t)$, $v^t\in N_C(x^t)$ and
\begin{equation}\label{eq 5.5}
{v^t}^TA^T(Ax^t-\eta^t)<-\gamma^t_2\Vert A^T(Ax^t-\eta^t)\Vert\Vert v^t\Vert
\end{equation}
for all $t\geq 1$. This implies in particular that $A^T(Ax^t-\eta^t)\neq 0$ and $v^t\neq 0$ for all $t\ge 1$. Also, observe from (i) that for all sufficiently large $t$, we have
\[
\Vert A^T(Ax^t-\eta^t)\Vert\geq \gamma_1\Vert Ax^t-\eta^t\Vert.
\]
Thus, by passing to subsequences if necessary, we may assume without loss of generality that
\begin{equation}\label{pqlimit}
\frac{Ax^t-\eta^t}{\Vert A^T(Ax^t-\eta^t)\Vert}\to p^*\mbox{ and } \frac{v^t}{\Vert v^t\Vert}\to q^*
\end{equation}
for some $p^*$ and $q^*$ so that $\|A^Tp^*\| = \|q^*\|= 1$.

Next, observe from $\eta^t\in{\rm Proj}_D(Ax^t)$, $x^t\to x^*$ and $Ax^*\in D$ that
\[
\limsup_{t\to\infty}\Vert\eta^t-Ax^t\Vert\leq \limsup_{t\to\infty}\Vert Ax^t-Ax^*\Vert = 0.
\]
Hence, $\eta^t\to Ax^*$. In addition, the relation $\eta^t\in {\rm Proj}_D(Ax^t)$ together with \eqref{normalcone} shows that for all $t\ge 1$,
\[
Ax^t - \eta^t \in N_D(\eta^t).
\]
This together with \eqref{pqlimit}, the fact that $\eta^t\to Ax^*$ and the closedness of the normal cone mapping gives $p^*\in N_D(Ax^*)$. Similarly, the relation $v^t\in N_C(x^t)$, \eqref{pqlimit}, the fact that $x^t\to x^*$ and the closedness of normal cone mapping imply $q^*\in N_C(x^*)$.
Now, divide both sides of \eqref{eq 5.5} by $\Vert A^T(Ax^t-\eta^t)\Vert\Vert v^t\Vert$ and pass to the limit,
we see that $(A^Tp^*)^Tq^*\leq -1$. Hence
\[
\Vert A^Tp^*+q^*\Vert^2=\|A^Tp^*\|^2+ \|q^*\|^2 +2(A^Tp^*)^Tq^*=2+2(A^Tp^*)^Tq^*\leq 0,
\]
where the second equality holds because of \eqref{pqlimit}. The above display shows that
$A^Tp^*+q^*=0$. This together with $p^*\in N_D(Ax^*)$, $q^*\in N_C(x^*)$ and \eqref{eq 5.3} gives $A^Tp^*=q^*=0$, which is a contradiction. This completes the proof.
\qed\end{proof}

We are now ready to show that under \eqref{eq 5.3}, the function $F$ defined in \eqref{eq 1.2} has the alleged KL property.
\begin{theorem}[{{\bf KL exponent under linear regularity}}]\label{thm:regimpKL}
Consider the split feasibility problem \eqref{eq 1.1}. Let $x^*\in C\cap A^{-1}D$ and suppose that the pair of sets $\{C,A^{-1}D\}$ has a linearly regular intersection with respect to $A$ at $x^*$, i.e., \eqref{eq 5.3} holds.
Then the function $F$ defined in \eqref{eq 1.2} has the KL property with exponent $\frac{1}{2}$ at $x^*$.
\end{theorem}
\begin{proof}
Let $\epsilon=\min\{\epsilon_1,\epsilon_2\}$, where $\epsilon_1$ and $\epsilon_2$ are as in Lemma~\ref{lemma 5.1}(i) and (ii) respectively. When $x\in B(x^*,\epsilon)\cap C$, $\eta\in{\rm Proj}_D(Ax)$ and $v\in N_C(x)$, we have
\begin{equation}\begin{split}\label{eq 5.8}
\Vert A^T(Ax-\eta)+v\Vert^2&=\Vert A^T(Ax-\eta)\Vert^2+\Vert v\Vert^2+2v^TA^T(Ax-\eta)\\
&\geq \Vert A^T(Ax-\eta)\Vert^2+\Vert v\Vert^2-2\gamma_2\Vert A^T(Ax-\eta)\Vert\Vert v\Vert\\
&\geq (1-\gamma_2)(\Vert A^T(Ax-\eta)\Vert^2+\Vert v\Vert^2),
\end{split}\end{equation}
where the first inequality follows from Lemma~\ref{lemma 5.1}(ii). Thus, whenever $x\in B(x^*,\epsilon)\cap C$, we have
\begin{equation*}
\begin{aligned}
&d(0,\partial F(x))\overset{\rm (a)}\geq \inf_{\eta\in {\rm Proj}_D(Ax),~v\in N_C(x)}\sqrt{\Vert A^T(Ax-\eta)+v\Vert^2}\\
&\overset{\rm (b)}\geq \sqrt{1-\gamma_2}\inf_{\eta\in {\rm Proj}_D(Ax),~v\in N_C(x)}\sqrt{\Vert A^T(Ax-\eta)\Vert^2+\Vert v\Vert^2}\\
&\geq \sqrt{1-\gamma_2}\inf_{\eta\in {\rm Proj}_D(Ax)}\Vert A^T(Ax-\eta)\Vert\overset{\rm (c)}\geq \gamma_1\sqrt{1-\gamma_2}\inf_{\eta\in {\rm Proj}_D(Ax)}\Vert Ax-\eta\Vert\\
&=\gamma_1\sqrt{2(1-\gamma_2)}(F(x)-F(x^*))^{\frac{1}{2}},
\end{aligned}
\end{equation*}
where (a) follows from Proposition~\ref{subF}, (b) follows from \eqref{eq 5.8} and (c) follows from Lemma~\ref{lemma 5.1}(i). This completes the proof.
\qed\end{proof}

\section{Factorizing completely positive matrices}\label{sec 6.1}
In this section, we consider the problem of factorizing completely positive matrices. Recall that a symmetric matrix $G\in{\mathbb R}^{n\times n}$ is completely positive if there exists a $B\in{\mathbb R}^{n\times r}_+$ for some $r\ge 1$ such that $G=BB^T$. It is known that determining whether a given matrix is completely positive is NP-hard; see, for example, \cite{DickGijb2014} and references therein.

Given a completely positive matrix $G$, the factorization problem aims at finding a $B\in {\mathbb{R}}^{n\times r}_+$ for some $r\ge 1$ so that $G = BB^T$.
In \cite{GroeDur2018}, this factorization problem was reformulated as a feasibility problem. Precisely, given a completely positive matrix $G\in \R^{n\times n}$, the authors in \cite{GroeDur2018} started with an initial factorization $G = BB^T$ for some $B\in \R^{n\times r}$; here, $B$ may not be entrywise nonnegative and $r\ge n$. They then rewrite the factorization as $G = (BQ)(BQ)^T$ for some orthogonal matrix $Q\in \R^{r\times r}$. If $r$ is chosen to be at least as large as the completely positive rank of $G$ (see \cite[Definition~2.2]{GroeDur2018}), then the completely positive matrix factorization problem is equivalent to finding an orthogonal matrix $Q$ so that $BQ$ is entrywise nonnegative, i.e.,
\begin{equation}\label{eq 7.12}
{\rm Find}~Q\in \R^{r\times r}~{\rm s.t.}~Q\in \mathcal{P}~{\rm and}~Q\in C,
\end{equation}
where $C$ is the set of $r\times r$ orthogonal matrices, and $\mathcal{P}:=\{Q\in {\mathbb R}^{r\times r}:\; BQ\in \R^{n\times r}_+\}$.
We would like to point out that the completely positive rank of $G$ is generally hard to compute (see \cite{berman2015open}) and we refer the readers to \cite[Theorem~4.1]{BomzeDickStill2015} for upper bounds of completely positive rank. These upper bounds are in the order of $n^2$ for large $n$. Instead of using these bounds as $r$, in our experiments, as a heuristic, we choose $r$ in the order of $n$ and we will specify our choices later.

In \cite{GroeDur2018}, the authors considered two algorithms for solving \eqref{eq 7.12}:
\begin{enumerate}
  \item the classical alternating projection method, which can be inefficient because ${\rm Proj}_{\cal P}$ is in general difficult to compute;
  \item the modified alternating projection algorithm (see \cite[Algorithm~2]{GroeDur2018}), which only requires computing projections onto $C$ and the nonnegative orthant $\mathbb{R}^{n\times r}_+$, as well as multiplications by $B$ and its Moore-Penrose inverse $B^\dagger$. This algorithm is described in Algorithm \ref{alg3} below.
\end{enumerate}
It was discussed in \cite[Section~5]{GroeDur2018} that the modified alternating projection algorithm is more efficient empirically than the classical alternating projection algorithm for solving \eqref{eq 7.12}.

\begin{algorithm}[h]
\caption{The modified alternating projection algorithm in \cite{GroeDur2018}}\label{alg3}
\begin{algorithmic}
\STATE {\bf Step 0.} Choose $r\geq n$ and $B\in{\mathbb R}^{n\times r}$ so that $G=BB^T$. Then select a $Q^0\in C$. Set $t=0$.

{\bf Step 1.} Compute $W^t = {\rm Proj}_{\mathbb{R}^{n\times r}_+}(BQ^t)$ and find
\[
Q^{t+1}\in{\rm Proj}_C[B^\dagger W^t+(I-B^\dagger B)Q^t].
\]

{\bf Step 2.} If a termination criterion is not met, set $t=t+1$. Go to {\bf Step 1}.

\end{algorithmic}
\end{algorithm}

Here, we consider an alternative approach for solving \eqref{eq 7.12}. Indeed, one can observe immediately that \eqref{eq 7.12} can be reformulated as the following split feasibility problem:
\begin{equation}\label{eq 7.11}
{\rm Find}~Q\in \R^{r\times r}~{\rm s.t.}~BQ\in D~{\rm and}~Q\in C,
\end{equation}
where $D=\mathbb{R}^{n\times r}_+$ and $C$ is the set of $r\times r$ orthogonal matrices. Note that the projections onto $C$ and $D$ have closed form solutions; see, for example, \cite[Lemma~4.1]{GroeDur2018} for the closed form formula of ${\rm Proj}_C$. Moreover, we have
$C^{\infty}=\{0\}$ because the set of $r\times r$ orthogonal matrices is bounded. Thus, we can apply ${\sf SpFeas}_{{\rm DC}_{\sf ls}}$ to solving \eqref{eq 7.11}, and any accumulation point of the sequence generated is a stationary point of the split feasibility problem \eqref{eq 7.11} according to Theorem~\ref{thm 3.1}.\footnote{We note that in this case $D$ is convex and hence $Q\mapsto \frac12d^2(BQ,D)$ is smooth. Our algorithm ${\sf SpFeas}_{{\rm DC}_{\sf ls}}$ reduces to the standard gradient projection algorithm with nonmonotone linesearch.}

\subsection{Numerical experiments for completely positive matrix factorization}
In this section, we compare ${\sf SpFeas}_{{\rm DC}_{\sf ls}}$ and the modified alternating projection algorithm (i.e., Algorithm~\ref{alg3})
for solving \eqref{eq 7.11} (or equivalently, \eqref{eq 7.12}). All codes are written in Matlab, and the experiments are performed in
Matlab 2019b on a 64-bit PC with an Intel(R) Core(TM) i7-6700 CPU (3.40GHz) and 32GB of RAM.

We first discuss the implementation details of the algorithms. In ${\sf SpFeas}_{{\rm DC}_{\sf ls}}$, we set $M=4$, $\tau=2$, $c=10^{-4}$, $L_{\rm max}=10^8$ and $L_{\rm min}=10^{-8}$. Moreover, we set $L^0_0=1$, and when $t\geq 1$:
\[
L^0_t=
\begin{cases}
{\rm min}\{{\rm max}\{\frac{{\rm tr}([Y^t]^TS^t)}{\Vert S^t\Vert_F^2}, L_{\rm min}\},L_{\rm max}\}& {\rm if}~{\rm tr}([Y^t]^TS^t)\geq 10^{-16}\\
{\rm min}\{{\rm max}\{\frac{{\bar L}_{t-1}}{1.1}, L_{\rm min}\},L_{\rm max}\}& \text{${\rm otherwise}.$}
\end{cases}
\]
where $S^t=Q^t-Q^{t-1}$ and $Y^t=B^T[BQ^t-\eta^t]-B^T[BQ^{t-1}-\eta^{t-1}]$, with $\eta^t\in {\rm Proj}_D(BQ^t)$ chosen in {\bf Step 1a)} of ${\sf SpFeas}_{{\rm DC}_{\sf ls}}$. We terminate it when ${\rm iter}>5000$ or ${\rm min}\{BQ^t\}_{ij}\geq -10^{-16}$ or $\bar L_t > 10^{10}$. On the other hand, for the modified alternating projection algorithm (i.e., Algorithm~\ref{alg3}), we terminate it when ${\rm iter}>5000$ or ${\rm min}\{BQ^t\}_{ij}\geq -10^{-15}$. We will describe their initializations later.

Both algorithms require a choice of $r\ge n$ and an initial factorization $G = BB^T$. In our experiments below, we follow the approach in \cite[Section~3]{GroeDur2018} to generate the $B$. Specifically,
given a completely positive matrix $G\in {\mathbb R}^{n\times n}$, we compute the Cholesky decomposition of $G$ such that $G=LL^T$ for some lower triangular matrix $L$, if successful, and set ${\bar B}=L$. On the other hand, if the Cholesky decomposition fails, we compute the eigenvalue decomposition of $G$ such that $G=U\Sigma_G U^T$ for some orthogonal matrix $U$ and diagonal matrix $\Sigma_G$, and set ${\bar B}=U\Sigma_G^{\frac{1}{2}}U^T$. Then we define $B$ as follows:
\begin{equation}\label{eq 7.13}
B:=\bigg[{\bar {\bf b}}_1, \dots,{\bar {\bf b}}_{j-1},\frac{1}{\sqrt m}{\bar {\bf b}}_j,{\bar {\bf b}}_{j+1}, \dots,{\bar {\bf b}}_{n}, \underbrace{\frac{1}{\sqrt m}{\bar {\bf b}}_j,\dots,\frac{1}{\sqrt m}{\bar {\bf b}}_j}_{m-1~{\rm columns}}\bigg],
\end{equation}
where ${\bar {\bf b}}_{j}$ is the column of ${\bar B}$ with the least number of negative entries, and $m=r-n+1$.

We perform two experiments comparing ${\sf SpFeas}_{{\rm DC}_{\sf ls}}$ and Algorithm~\ref{alg3}.
In our first experiment, we consider randomly generated completely positive matrices as in \cite[Section~7.8]{GroeDur2018}. We generate a random $n\times n$ completely positive matrix $G$ using the following MATLAB code:
\begin{verbatim}
G_0 = abs(randn(n,2*n));  G = G_0*G_0';
\end{verbatim}
We set $(n,r)=(n,1.5n)=(10i,15i)$ for $i=1,2,3,4,10,20,30,40$ in Table \ref{table 7.1} and \ref{table 7.15}, and set $ (n, r) = (n, 3n+1) = (10i, 30i+1) $ for $ i=1,2,3,4,10,20 $ in Table \ref{table 7.16}. For each $i$, we randomly generate $50$ completely positive matrices $G$ as described above. We generate $B$ as in \eqref{eq 7.13} and consider two possible ways of initializing the algorithms:
\begin{enumerate}[(a)]
  \item We initialize both algorithms at $Q^0 = I$ for solving the corresponding \eqref{eq 7.11}.
  \item We initialize both algorithms at the same random initial point, where we first generate an $r\times r$ matrix $\widetilde Q$ with i.i.d. standard Gaussian entries and then pick any $Q^0\in {\rm Proj}_C(\widetilde Q)$.
\end{enumerate}
We first present the computational results with $Q^0 = I$ in Table~\ref{table 7.1}, where we report the
largest and smallest function values ($\frac12 d^2(BQ^t,D)$) at termination, the average number of iterations among successful instances (${\rm iter}_{\rm s}$) and the average number of iterations among failed instances (${\rm iter}_{\rm f}$).\footnote{\label{notesuccess}We say that the instance is solved by the algorithm successfully if the algorithm is terminated with the desired accuracy achieved, i.e, ${\rm min}\{BQ^t\}_{ij}\geq -10^{-16}$ for ${\sf SpFeas}_{{\rm DC}_{\sf ls}}$, and ${\rm min}\{BQ^t\}_{ij}\geq -10^{-15}$ for Algorithm~\ref{alg3}.} We also report the average CPU time (in seconds) among successful instances (${\rm CPU}_{\rm s}$) and the average CPU time among failed instances (${\rm CPU}_{\rm f}$). The success rate is also listed. We can see from Table~\ref{table 7.1} that ${\sf SpFeas}_{{\rm DC}_{\sf ls}}$ significantly outperforms Algorithm~\ref{alg3}, with ${\sf SpFeas}_{{\rm DC}_{\sf ls}}$ being able to solve all instances and being much faster.

\begin{table}[h]
\centering
\caption{Comparing ${\sf SpFeas}_{{\rm DC}_{\sf ls}}$ and Algorithm \ref{alg3} on factorizing random completely positive matrices when $Q^0 = I$.}
\begin{tabular}{|c|c|c|c|c|c|c|c|c|}
\hline
&&\multicolumn{7}{c|}{${\sf SpFeas}_{{\rm DC}_{\sf ls}}$}\\ \cline{3-9}
$n$ &$r$& success~($\%$) & ${\rm fval}_{\rm max}$ & ${\rm fval}_{\rm min}$ & ${\rm iter}_{\rm s}$ &${\rm iter}_{\rm f}$ & ${\rm CPU}_{\rm s}$& ${\rm CPU}_{\rm f}$ \\ \hline
 10  &    15  &   100 & 2e-33 & 0e+00 &     5 &   - & 0.0010  &   -  \\
 20  &    30  &   100 & 0e+00 & 0e+00 &     8 &   - & 0.0014  &   -  \\
 30  &    45  &   100 & 0e+00 & 0e+00 &    10 &   - & 0.0050  &   -  \\
 40  &    60  &   100 & 4e-33 & 0e+00 &    11 &   - & 0.0077  &   -  \\
100  &   150  &   100 & 0e+00 & 0e+00 &    18 &   - & 0.0622  &   -  \\
200  &   300  &   100 & 3e-33 & 0e+00 &   190 &   - & 2.2804  &   -  \\
300  &   450  &   100 & 5e-33 & 0e+00 &   486 &   - & 14.0715 &   -  \\
400  &   600  &   100 & 1e-34 & 0e+00 &   731 &   - & 43.5873 &   -  \\
\hline
&&\multicolumn{7}{c|}{Algorithm ~\ref{alg3}}\\ \cline{3-9}
$n$ &$r$& success~($\%$) & ${\rm fval}_{\rm max}$ & ${\rm fval}_{\rm min}$ & ${\rm iter}_{\rm s}$ &${\rm iter}_{\rm f}$ & ${\rm CPU}_{\rm s}$& ${\rm CPU}_{\rm f}$ \\ \hline
10   &    15  &     0 & 1e+00 & 9e-03 &   - &  5001 &   - & 0.1195   \\
20   &    30  &     0 & 4e+00 & 1e-01 &   - &  5001 &   - & 0.5320   \\
30   &    45  &     0 & 6e+00 & 2e-01 &   - &  5001 &   - & 1.6316   \\
40   &    60  &     0 & 6e+00 & 6e-01 &   - &  5001 &   - & 2.3411   \\
100  &   150  &     0 & 1e+01 & 2e+00 &   - &  5001 &   - & 13.9835  \\
200  &   300  &     0 & 3e+01 & 1e+01 &   - &  5001 &   - & 57.6484  \\
300  &   450  &     0 & 3e+02 & 4e+01 &   - &  5001 &   - & 147.4416 \\
400  &   600  &     0 & 3e+02 & 9e+01 &   - &  5001 &   - & 293.6904 \\
\hline
\end{tabular}
\label{table 7.1}
\end{table}

Then, in Table~\ref{table 7.15} and Table~\ref{table 7.16}, we present the computational results with random initial points, and set $ r = 1.5n $ and $ r=3n+1 $ respectively.\footnote{We do not present results with $ (n, 3n+1) = (10i, 30i+1) $ for $ i=30, 40 $ because they take too much CPU time.} Here, for each random instance, we run the algorithms on the same set of random initial points, where we use at most $100$ initial points when $n\le 50$, and at most $10$ initial points otherwise,\footnote{This choice follows the one used in \cite[Table~2]{GroeDur2018}.} and
declare a success once the random instance is solved by the algorithm successfully.\footnoteref{notesuccess} We report the
largest and smallest function values ($\frac12 d^2(BQ^t,D)$) at termination, the average total number of iterations among successful instances (${\rm iter}_{\rm s}$) and failed instances (${\rm iter}_{\rm f}$), the average total CPU time (in seconds) among successful instances (${\rm CPU}_{\rm s}$) and failed instances (${\rm CPU}_{\rm f}$), and the success rate.
We also report the average number of random initial points used among instances that are successfully solved (${\rm InitNo_s}$). We can see from Table~\ref{table 7.15} and Table~\ref{table 7.16} that ${\sf SpFeas}_{{\rm DC}_{\sf ls}}$ significantly outperforms Algorithm~\ref{alg3}, with ${\sf SpFeas}_{{\rm DC}_{\sf ls}}$ being able to solve all instances using only {\em one} random initial point, and being much faster. Moreover, by comparing these two tables, we can see that ${\sf SpFeas}_{{\rm DC}_{\sf ls}}$ performs similarly for the two choices of $ r $, while the performance of Algorithm~\ref{alg3} is sensitive to the choice of $ r $.

\begin{table}[h]
	\centering
	\caption{Comparing ${\sf SpFeas}_{{\rm DC}_{\sf ls}}$ and Algorithm \ref{alg3} on factorizing random completely positive matrices with $ r = 1.5n $ and random initializations.}
	\begin{tabular}{|c|c|c|c|c|c|c|c|c|c|}
		\hline
		&&\multicolumn{8}{c|}{${\sf SpFeas}_{{\rm DC}_{\sf ls}}$}\\ \cline{3-10}
		$n$ &$r$& success~($\%$) & ${\rm fval}_{\rm max}$ & ${\rm fval}_{\rm min}$ & ${\rm iter}_{\rm s}$ &${\rm iter}_{\rm f}$ & ${\rm CPU}_{\rm s}$ & ${\rm CPU}_{\rm f}$ & $ {\rm InitNo_s} $  \\ \hline
         10   &    15  &   100 & 0e+00 & 0e+00 &    14 &   - & 0.0013  &   - &   1.0 \\
         20   &    30  &   100 & 0e+00 & 0e+00 &    18 &   - & 0.0026  &   - &   1.0 \\
         30   &    45  &   100 & 0e+00 & 0e+00 &    25 &   - & 0.0114  &   - &   1.0 \\
         40   &    60  &   100 & 0e+00 & 0e+00 &    37 &   - & 0.0212  &   - &   1.0 \\
         100  &   150  &   100 & 0e+00 & 0e+00 &   112 &   - & 0.3440  &   - &   1.0 \\
         200  &   300  &   100 & 4e-33 & 0e+00 &   243 &   - & 3.0006  &   - &   1.0 \\
         300  &   450  &   100 & 4e-33 & 0e+00 &   463 &   - & 14.7035 &   - &   1.0 \\
         400  &   600  &   100 & 3e-33 & 0e+00 &   711 &   - & 44.6827 &   - &   1.0 \\
		\hline
		&&\multicolumn{8}{c|}{Algorithm ~\ref{alg3}}\\ \cline{3-10}
		$n$ &$r$& success~($\%$) & ${\rm fval}_{\rm max}$ & ${\rm fval}_{\rm min}$ & ${\rm iter}_{\rm s}$ &${\rm iter}_{\rm f}$ & ${\rm CPU}_{\rm s}$ & ${\rm CPU}_{\rm f}$ & $ {\rm InitNo_s} $  \\ \hline
    10   &    15  &   100 & 2e-30 & 2e-33 &  5232 &   -   & 0.1474  &   -       &   1.9  \\
    20   &    30  &   100 & 3e-30 & 9e-32 &  4468 &   -   & 0.4483  &   -       &   1.5  \\
    30   &    45  &   100 & 2e-30 & 2e-33 &  5496 &   -   & 2.0830  &   -       &   1.7  \\
    40   &    60  &   100 & 2e-30 & 2e-32 &  7497 &   -   & 3.4731  &   -       &   1.9  \\
    100  &   150  &     2 & 2e-06 & 1e-31 & 44693 & 50010 & 99.1045 & 117.2377  &   9.0  \\
    200  &   300  &     0 & 4e-05 & 2e-28 &   -   & 50010 &   -     & 546.9136  &   -    \\
    300  &   450  &     0 & 2e-07 & 2e-28 &   -   & 50010 &   -     & 1413.3341 &   -    \\
    400  &   600  &     0 & 7e-04 & 1e-27 &   -   & 50010 &   -     & 2809.6050 &   -    \\
		\hline
	\end{tabular}	
\label{table 7.15}
\end{table}	

\begin{table}[h]
	\centering
	\caption{Comparing ${\sf SpFeas}_{{\rm DC}_{\sf ls}}$ and Algorithm \ref{alg3} on factorizing random completely positive matrices with $ r = 3n+1 $ and random initializations.}
	\begin{tabular}{|c|c|c|c|c|c|c|c|c|c|}
		\hline
		&&\multicolumn{8}{c|}{${\sf SpFeas}_{{\rm DC}_{\sf ls}}$}\\ \cline{3-10}
		$n$ &$r$& success~($\%$) & ${\rm fval}_{\rm max}$ & ${\rm fval}_{\rm min}$ & ${\rm iter}_{\rm s}$ &${\rm iter}_{\rm f}$ & ${\rm CPU}_{\rm s}$ & ${\rm CPU}_{\rm f}$ & $ {\rm InitNo_s} $  \\ \hline
		 10  &    31  &   100 & 4e-35 & 0e+00 &    13 & - & 0.0030  & - &   1.0 \\
		 20  &    61  &   100 & 0e+00 & 0e+00 &    14 & - & 0.0070  & - &   1.0 \\
		 30  &    91  &   100 & 0e+00 & 0e+00 &    16 & - & 0.0168  & - &   1.0 \\
		 40  &   121  &   100 & 0e+00 & 0e+00 &    19 & - & 0.0303  & - &   1.0 \\
		100  &   301  &   100 & 1e-35 & 0e+00 &    93 & - & 1.0123  & - &   1.0 \\
		200  &   601  &   100 & 5e-33 & 0e+00 &   209 & - & 11.0123 & - &   1.0 \\
		\hline
		&&\multicolumn{8}{c|}{Algorithm ~\ref{alg3}}\\ \cline{3-10}
		$n$ &$r$& success~($\%$) & ${\rm fval}_{\rm max}$ & ${\rm fval}_{\rm min}$ & ${\rm iter}_{\rm s}$ &${\rm iter}_{\rm f}$ & ${\rm CPU}_{\rm s}$ & ${\rm CPU}_{\rm f}$ & $ {\rm InitNo_s} $  \\ \hline
		 10  &    31  &   100 & 4e-30 & 3e-31 &  1237 &   -   & 0.1126 &   -       & 1.1 \\
		 20  &    61  &   100 & 3e-30 & 2e-32 &   753 &   -   & 0.2883 &   -       & 1.0 \\
		 30  &    91  &   100 & 4e-30 & 2e-31 &  2065 &   -   & 1.6742 &   -       & 1.0 \\
		 40  &   121  &   100 & 5e-30 & 2e-31 &  4739 &   -   & 6.0358 &   -       & 1.4 \\
		100  &   301  &     0 & 3e-27 & 2e-29 &   -   & 50010 &   -    & 458.5531  & - 	 \\
		200  &   601  &     0 & 6e-27 & 2e-28 &   -   & 50010 &   -    & 2300.6188 & -   \\
		\hline
	\end{tabular}	
	\label{table 7.16}
\end{table}	

Next, as in \cite[Section~7.6]{GroeDur2018}, we perform a second experiment to study the performance of the algorithms in factorizing completely positive matrices that are close to the boundary of the completely positive cone. Specifically, as in \cite[Example~7.3]{GroeDur2018}, we consider
\[
G:=\begin{bmatrix}
8&5&1&1&5\\
5&8&5&1&1\\
1&5&8&5&1\\
1&1&5&8&5\\
5&1&1&5&8
\end{bmatrix},~P:=\begin{bmatrix}
     2   &  1  &   1   &  1 &    1\\
     1   &  2   &  1   &  1    & 1\\
     1    & 1 &    2   &  1   &  1\\
     1&     1    & 1   &  2  &   1\\
     1 &    1    & 1   &  1  &   2\\
     \end{bmatrix},
\]
and define
\begin{equation}\label{G_lambda}
G_{\lambda}=\lambda G+(1-\lambda)P.
\end{equation}
We apply the two algorithms to factorizing $G_{\lambda}$ with different values of $\lambda$. Moreover, for the two algorithms, we consider random initializations: we first generate an $r\times r$ matrix $\widetilde Q$ with i.i.d. standard Gaussian entries and then pick any $Q^0\in {\rm Proj}_C(\widetilde Q)$.

In our experiments below, we consider $\lambda$ as listed in Table~\ref{table 7.2} and set $r = 12$: This choice of $r$ was also used in \cite[Section~7.6]{GroeDur2018}. For each $\lambda$, we consider $100$ random initializations as described above, and apply the two algorithms to factorizing $G_\lambda$ from these initial points. Our computational results are presented in Table~\ref{table 7.2}, where we report the
largest and smallest function values ($\frac12 d^2(BQ^t,D)$) at termination, the average number of iterations among successful instances (${\rm iter}_{\rm s}$) and the average number of iterations among failed instances (${\rm iter}_{\rm f}$).\footnote{As in the previous experiment, we say that the instance is solved by the algorithm successfully if the algorithm is terminated with the desired accuracy achieved, i.e, ${\rm min}\{BQ^t\}_{ij}\geq -10^{-16}$ for ${\sf SpFeas}_{{\rm DC}_{\sf ls}}$, and ${\rm min}\{BQ^t\}_{ij}\geq -10^{-15}$ for Algorithm~\ref{alg3}.} We also report the average CPU time (in seconds) among successful instances (${\rm CPU}_{\rm s}$) as well as the average CPU time among failed instances (${\rm CPU}_{\rm f}$). We can see from Table \ref{table 7.2} that ${\sf SpFeas}_{{\rm DC}_{\sf ls}}$ again significantly outperforms Algorithm~\ref{alg3}. Moreover, the success rates for both algorithms decrease when $\lambda$ increases.

\begin{table}[h]
\centering
\caption{Comparing ${\sf SpFeas}_{{\rm DC}_{\sf ls}}$ and Algorithm \ref{alg3} on factorizing $G_\lambda$ in \eqref{G_lambda}.}
\begin{tabular}{|c|c|c|c|c|c|c|c|c|}
\hline
&&\multicolumn{7}{c|}{${\sf SpFeas}_{{\rm DC}_{\sf ls}}$}\\ \cline{3-9}
$\lambda$ &$r$& success~($\%$) & ${\rm fval}_{\rm max}$ & ${\rm fval}_{\rm min}$ & ${\rm iter}_{\rm s}$ &${\rm iter}_{\rm f}$ & ${\rm CPU}_{\rm s}$& ${\rm CPU}_{\rm f}$\\ \hline
    0.00  &    12  &   100 & 0e+00 & 0e+00 &     8 &   -   & 0.0011 &   -     \\
    0.20  &    12  &   100 & 0e+00 & 0e+00 &    27 &   -   & 0.0011 &   -     \\
    0.40  &    12  &   100 & 0e+00 & 0e+00 &    87 &   -   & 0.0031 &   -     \\
    0.60  &    12  &   100 & 0e+00 & 0e+00 &   333 &   -   & 0.0116 &   -     \\
    0.80  &    12  &    99 & 3e-21 & 0e+00 &  1229 &  5001 & 0.0399 & 0.1510  \\
    0.90  &    12  &    82 & 2e-17 & 0e+00 &  2547 &  5001 & 0.0786 & 0.1529  \\
    0.95  &    12  &    29 & 7e-13 & 0e+00 &  2897 &  5001 & 0.0907 & 0.1611  \\
    0.96  &    12  &     9 & 4e-11 & 0e+00 &  2994 &  5001 & 0.0951 & 0.1656  \\
    0.97  &    12  &     4 & 7e-12 & 0e+00 &  2892 &  5001 & 0.0911 & 0.1709  \\
    0.98  &    12  &     5 & 2e-10 & 0e+00 &  3356 &  5001 & 0.1057 & 0.1814  \\
    0.99  &    12  &     1 & 8e-07 & 0e+00 &  3421 &  5001 & 0.1243 & 0.2047  \\
\hline
&&\multicolumn{7}{c|}{Algorithm ~\ref{alg3}}\\ \cline{3-9}
$\lambda$ &$r$& success~($\%$) & ${\rm fval}_{\rm max}$ & ${\rm fval}_{\rm min}$ & ${\rm iter}_{\rm s}$ &${\rm iter}_{\rm f}$ & ${\rm CPU}_{\rm s}$& ${\rm CPU}_{\rm f}$\\ \hline
0.00  &    12  &    99 & 5e-17 & 2e-31 &   155 &  5001 & 0.0047 & 0.1158  \\
0.20  &    12  &    91 & 2e-10 & 2e-32 &   476 &  5001 & 0.0114 & 0.0955  \\
0.40  &    12  &    77 & 5e-02 & 2e-31 &   703 &  5001 & 0.0161 & 0.0937  \\
0.60  &    12  &    45 & 2e-01 & 3e-32 &   871 &  5001 & 0.0198 & 0.0877  \\
0.80  &    12  &    36 & 4e-01 & 7e-32 &  1240 &  5001 & 0.0268 & 0.0889  \\
0.90  &    12  &    16 & 5e-01 & 4e-31 &  2899 &  5001 & 0.0558 & 0.0893  \\
0.95  &    12  &    26 & 6e-01 & 4e-31 &  3106 &  5001 & 0.0590 & 0.0879  \\
0.96  &    12  &    31 & 6e-01 & 2e-31 &  2842 &  5001 & 0.0547 & 0.0895  \\
0.97  &    12  &    36 & 6e-01 & 2e-31 &  2980 &  5001 & 0.0577 & 0.0870  \\
0.98  &    12  &     0 & 7e-01 & 8e-17 &   -   &  5001 &   -    & 0.0921  \\
0.99  &    12  &     0 & 7e-01 & 1e-04 &   -   &  5001 &   -    & 0.0922  \\
\hline
\end{tabular}
\label{table 7.2}
\end{table}

\section{Sparse matrix factorization}\label{sec7}

Given a matrix $G$, the sparse matrix factorization problem consists in factorizing $G$ (approximately) as the product of several sparse matrices. This problem is closely related to deep learning, sparse encoding and dictionary learning; see \cite{NeysPani2014} and references therein.
In this section, we consider a special instance of the sparse matrix factorization problem. Specifically, given a positive semidefinite matrix $G\in \R^{n\times n}$, we would like to find a {\em sparse} matrix $P\in \R^{n\times n}$ so that $G= PP^T$. In addition, we require the {\em columns} of $P$ to be uniformly sparse: this ensures the cost of the multiplication ${\bf p}^T_jx$ remains more or less the same for each $j$, where ${\bf p}_j$ is the $j$th column of $P$ and $x$ is an $n$-dimensional vector. More precisely,
our problem is described as follows:
 \begin{equation}\label{eq 8.14}
{\rm Find}~P\in \R^{n\times n}~{\rm s.t.}~\max_{j=1,\dots,n}\Vert {\bf p}_j\Vert_0\leq s~{\rm and}~G=PP^T,
\end{equation}
where $G\in {\R}^{n\times n}$ is a given positive semidefinite matrix, and $\|v\|_0$ is the number of nonzero entries of the vector $v$.

To solve \eqref{eq 8.14}, we mimic the approach described in Section \ref{sec 6.1} and reformulate it as a split feasibility problem. In detail, starting with an initial factorization $G=BB^T$ for some $B\in \mathbb{R}^{n\times n}$, one can see that \eqref{eq 8.14} can be equivalently reformulated as the following split feasibility problem:
\begin{equation}\label{eq 8.15}
{\rm Find}~Q\in \R^{n\times n}~{\rm s.t.}~BQ\in D~{\rm and}~Q\in C,
\end{equation}
where $C$ is the set of $n\times n$ orthogonal matrices and $D=\{U\in {\mathbb R}^{n\times n}:\;\Vert {\bf u}_j\Vert_0\leq s\ \mbox{for each }i=1,\ldots,n\}$, with ${\bf u}_j$ being the $j$th column of $U\in \R^{n\times n}$.
It is easy to see that if $Q^*$ is a solution of \eqref{eq 8.15}, then ${BQ^*}$ solves \eqref{eq 8.14}.

Note that the projections onto $C$ and $D$ have closed form solutions; see \cite[Lemma~4.1]{GroeDur2018} and \cite[Proposition~3.1]{LuZhan2013} for the closed form formula of ${\rm Proj}_C$ and ${\rm Proj}_D$, respectively. In addition, the boundedness of $C$ implies that
$C^{\infty}=\{0\}$. Therefore, we can employ ${\sf SpFeas}_{{\rm DC}_{\sf ls}}$ and ${\sf SpFeas}_{\rm DC}$ to solve \eqref{eq 8.15} according to the discussions in Section~\ref{sec3}, and it follows from Theorem~\ref{thm 3.1} and Corollary~\ref{coro 4.1} that any accumulation point of the sequence generated is a stationary point of the split feasibility problem \eqref{eq 8.15}.

\subsection{Numerical experiments for sparse matrix factorization}
In this section, we perform numerical experiments to compare the performances of ${\sf SpFeas}_{{\rm DC}_{\sf ls}}$ and ${\sf SpFeas}_{{\rm DC}}$ on solving \eqref{eq 8.15}.
All codes are written in Matlab, and the experiments are performed in
Matlab 2019b on a 64-bit PC with an Intel(R) Core(TM) i7-6700 CPU (3.40GHz) and 32GB of RAM.

We first discuss the implementation details of the algorithms. In ${\sf SpFeas}_{{\rm DC}_{\sf ls}}$, we set $M=4$, $\tau=2$, $c=10^{-4}$, $L_{\rm max}=10^8$, $L_{\rm min}=10^{-8}$. Moreover, we set $L^0_0=1$, and when $t\geq 1$:
$$L^0_t=
\begin{cases}
{\rm min}\{{\rm max}\{\frac{{\rm tr}([Y^t]^TS^t)}{\Vert S^t\Vert_F^2}, L_{\rm min}\},L_{\rm max}\}& {\rm if}~{\rm tr}([Y^t]^TS^t)\geq 10^{-12}\\
{\rm min}\{{\rm max}\{\frac{{\bar L}_{t-1}}{2.5}, L_{\rm min}\},L_{\rm max}\}& \text{${\rm otherwise}.$}
\end{cases}$$
where $S^t=Q^t-Q^{t-1}$ and $Y^t=B^T[BQ^t-\eta^t]-B^T[BQ^{t-1}-\eta^{t-1}]$, with $\eta^t\in {\rm Proj}_D(BQ^t)$ chosen in {\bf Step 1a)} of ${\sf SpFeas}_{{\rm DC}_{\sf ls}}$. On the other hand, for ${\sf SpFeas}_{\rm DC}$, we set $L=\lambda_{\max}(B^TB)+10^{-4}$. We initialize both algorithms at the identity matrix. We terminate ${\sf SpFeas}_{{\rm DC}_{\sf ls}}$ when $d(BQ^t,D)<10^{-9}$, or ${\rm iter}>10000$, or $L_t > 10^{10}$, while ${\sf SpFeas}_{\rm DC}$ is terminated when $d(BQ^t,D)<10^{-9}$ or ${\rm iter}>10000$.

We compare the above algorithms on randomly generated positive semidefinite matrices that admit uniformly sparse factorizations. We first generate a random matrix $P_0\in \R^{n\times n}$ with i.i.d. standard Gaussian entries. We then project $P_0$ onto $D$ to obtain $\tilde P_0$ and form a positive semidefinite matrix $G$ by $G = \tilde P_0\tilde P_0^T$. We then use the above algorithms to solve the corresponding \eqref{eq 8.15} with $B:= G^\frac12$.

In our experiment below, we set $(n,r)=(100i,0.1j)$ for $i=1,2,3,4$ and $j=6,7,8,9$ and set $s = nr$. For each $i$ and $j$, we generate $20$ random instances as described above. We present the computational results in Table~\ref{table 8.1}, where we report the value $d(BQ^t,D)$ at termination, the number of iterations (iter), and the CPU time in seconds (CPU), averaged over the $20$ random instances. One can see that ${\sf SpFeas}_{{\rm DC}_{\sf ls}}$ notably outperforms ${\sf SpFeas}_{\rm DC}$ in terms of both CPU times and the terminating function values. Moreover, both algorithms become faster when $s$ increases.

\begin{table}[!htbp]
\centering
\caption{Comparing ${\sf SpFeas}_{{\rm DC}_{\sf ls}}$ and ${\sf SpFeas}_{{\rm DC}}$ on solving \eqref{eq 8.15} with $s = nr$.}
\begin{tabular}{|c|c|c|c|c|c|c|c|}
\hline
&&\multicolumn{3}{c|}{${\sf SpFeas}_{{\rm DC}_{\sf ls}}$}&\multicolumn{3}{c|}{${\sf SpFeas}_{\rm DC}$}\\ \cline{3-8}
{\em n} & {\em r} & $d(BQ^t,D)$ & iter & CPU & $d(BQ^t,D)$ & iter & CPU\\ \hline
   100 &   0.6 &1e-09 &  4014 &  14.7 &3e-04 & 10001 &  17.8  \\
   100 &   0.7 &1e-09 &   774 &   2.7 &1e-09 &  7270 &  12.5  \\
   100 &   0.8 &9e-10 &   245 &   0.8 &1e-09 &  2940 &   5.0  \\
   100 &   0.9 &7e-10 &   141 &   0.4 &1e-09 &  1381 &   2.4  \\ \hline
   200 &   0.6 &1e-09 &  3800 &  54.0 &3e-04 & 10001 &  70.1  \\
   200 &   0.7 &1e-09 &   761 &  10.4 &1e-09 &  7526 &  50.7  \\
   200 &   0.8 &9e-10 &   241 &   3.0 &1e-09 &  3122 &  21.0  \\
   200 &   0.9 &7e-10 &   146 &   1.8 &1e-09 &  1477 &   9.9  \\ \hline
   300 &   0.6 &1e-09 &  4067 & 136.4 &4e-04 & 10001 & 155.2  \\
   300 &   0.7 &1e-09 &   734 &  22.8 &1e-09 &  7636 & 116.6  \\
   300 &   0.8 &9e-10 &   274 &   8.2 &1e-09 &  3168 &  48.6  \\
   300 &   0.9 &7e-10 &   149 &   4.1 &1e-09 &  1509 &  22.2  \\ \hline
   400 &   0.6 &1e-09 &  3736 & 231.8 &6e-04 & 10001 & 295.9  \\
   400 &   0.7 &1e-09 &   626 &  37.1 &1e-09 &  7711 & 222.5  \\
   400 &   0.8 &9e-10 &   246 &  13.8 &1e-09 &  3201 &  92.4  \\
   400 &   0.9 &7e-10 &   145 &   7.7 &1e-09 &  1528 &  43.9  \\ \hline
\end{tabular}
\label{table 8.1}
\end{table}

\section{An outlier detection problem}\label{sec8}

In this section, we look at an outlier detection problem. Specifically, we consider the problem of finding an $s$-sparse solution of a linear system $Ax = b$ where some of the $b_i$'s are wrongly recorded.
This class of problem arises in applications such as compressed sensing, where signals may be contaminated by the so-called electromyographic noise, resulting in extreme measurements \cite{PolaCarrBlanBarn2012}.

Here, we approach this problem by considering the following split feasibility problem, which assumes prior knowledge of $s$ and the number of incorrect $b_i$'s:
 \begin{equation}\label{eq 6.12}
{\rm Find}~x\in \R^{n}~{\rm s.t.}~x\in C,\ Ax\in D,
\end{equation}
where $A\in{\mathbb R}^{m\times n}$, $b\in {\R}^m$, $C=\{x\in \R^n:\Vert x\Vert_0\leq s, \Vert x\Vert_{\infty}\leq 10^8\}$, and $D=\{y\in \R^m:\Vert y\Vert_0\leq r\}+b$, with $r$ being an upper estimate of the number of incorrect $b_i$'s (outliers).

Note that the projections onto $C$ and $D$ have closed form solutions; see, for example, \cite[Proposition~3.1]{LuZhan2013}. Moreover, since $C$ is bounded, we have
$C^{\infty}=\{0\}$. Thus, we can apply ${\sf SpFeas}_{{\rm DC}_{\sf ls}}$ and ${\sf SpFeas}_{\rm DC}$ to solving \eqref{eq 6.12} according to the discussions in Section~\ref{sec3}. Moreover, any accumulation point of the sequence generated is a stationary point of the split feasibility problem \eqref{eq 6.12}.

\subsection{Numerical experiments for outlier detection}
In this section, we perform numerical experiments to study the behavior of  ${\sf SpFeas}_{{\rm DC}_{\sf ls}}$ and ${\sf SpFeas}_{\rm DC}$ on the outlier detection problem \eqref{eq 6.12}. All codes are written in Matlab, and the experiments are performed in
Matlab 2019b on a 64-bit PC with an Intel(R) Core(TM) i7-6700 CPU (3.40GHz) and 32GB of RAM.

We first discuss the implementation details of the algorithms. In ${\sf SpFeas}_{{\rm DC}_{\sf ls}}$, we set $M=4$, $\tau=2$, $c=10^{-4}$, $L_{\rm max}=10^8$ and $L_{\rm min}=10^{-8}$. We also set
$L^0_0=1$, and when $t\geq 1$,
\[
L^0_t=
\begin{cases}
{\rm min}\{{\rm max}\{\frac{{s^t}^Ty^t}{\Vert s^t\Vert^2}, L_{\rm min}\},L_{\rm max}\}& {\rm if}~{s^t}^Ty^t\geq 10^{-12}\\
{\rm min}\{{\rm max}\{\frac{{\bar L}_{t-1}}{2}, L_{\rm min}\},L_{\rm max}\}& \text{${\rm otherwise}.$}
\end{cases}
\]
where $s^t=x^t-x^{t-1}$, $y^t=A^T[Ax^t-\eta^t]-A^T[Ax^{t-1}-\eta^{t-1}]$, with $\eta^t$ defined in {\bf Step 1a)} of the algorithm. We initialize
${\sf SpFeas}_{{\rm DC}_{\sf ls}}$ at $x^0=0$ and terminate it when
\[
\frac{\sqrt{\left(\sqrt{\lambda_{\max}(A^TA)}\| A(x^t-x^{t-1})\|+{\bar L}_{t-1}\Vert x^t-x^{t-1}\Vert\right)^2+ \Vert x^t-x^{t-1}\Vert^2}}{{\rm max}\{1,\Vert x^t\Vert\}}<10^{-8};
\]
following the discussions in \cite[Section~6]{LiuPongTake2018}, this guarantees $d(0,\partial \Xi(x^t,\eta^{t-1}))< 10^{-8}{\rm max}\{1,\Vert x^t\Vert\}$, where $\Xi(x,\eta) := h(x) + P(x) - \eta^Tx + g^*(\eta)$, with $h$, $P$ and $g$ given in \eqref{eq 3.7}, and $g^*$ is the convex conjugate of $g$.\footnote{As discussed in \cite[Section~6]{LiuPongTake2018}, this termination criterion is motivated by the fact that \eqref{inclusion:immed} holds if and only if $0 \in \partial \Xi(\bar x,\bar \eta)$ for some $\bar \eta$.}
On the other hand, for ${\sf SpFeas}_{\rm DC}$, we set $L=\lambda_{\max}(A^TA)+10^{-4}$. We initialize this algorithm at $x^0=0$ and terminate it
when
\[
\frac{\sqrt{\left(\sqrt{\lambda_{\max}(A^TA)}\Vert A(x^t-x^{t-1})\Vert+L\Vert x^t-x^{t-1}\Vert\right)^2 + \Vert x^t-x^{t-1}\Vert^2}}{{\rm max}\{1,\Vert x^t\Vert\}}<10^{-8},
\]
or when the number of iterations reaches $3000$.

We compare the above algorithms on randomly generated instances. We first generate an $m\times n$ matrix $A$ with i.i.d.
standard Gaussian entries, and normalize it to have unit column norms.
We next generate an $s$-sparse vector $w\in{\mathbb R}^n$ with i.i.d. standard Gaussian entries at uniformly randomly chosen positions.
We then set
\[
b_i=
\begin{cases}
(Aw)_i\\
(Aw)_i+10\cdot {\rm sign}({\bar n}_{i-m+r})& i=m-r+1,\dots,m.
\end{cases}
\]
where ${\bar n}\in \R^r$ has i.i.d. standard Gaussian entries.

In our experiment below, we set $(n,m,s,r)$ as listed in Table~\ref{table 6.1}. For each quadruple $(n,m,s,r)$, we generate $20$ random instances as described above.
Our computational results are presented in Table~\ref{table 6.1}, where we report the
value $d(Ax^t,D)$ at termination, the number of iterations (iter), and the CPU time in seconds (CPU),
averaged over the $20$ random instances. One can observe that ${\sf SpFeas}_{{\rm DC}_{\sf ls}}$ significantly outperforms ${\sf SpFeas}_{\rm DC}$ in terms of both CPU times and the terminating function values.

\begin{table}[!htbp]
\centering
\caption{Comparing ${\sf SpFeas}_{{\rm DC}_{\sf ls}}$ and ${\sf SpFeas}_{{\rm DC}}$ on solving \eqref{eq 6.12}.}
\begin{tabular}{|c|c|c|c|c|c|c|c|c|c|}
\hline
&&&&\multicolumn{3}{c|}{${\sf SpFeas}_{{\rm DC}_{\sf ls}}$}&\multicolumn{3}{c|}{${\sf SpFeas}_{\rm DC}$}\\ \cline{5-10}
{\em n} & {\em m} & {\em s} & {\em r} & $d(Ax^t,D)$ & iter & CPU & $d(Ax^t,D)$ & iter & CPU\\ \hline
10000 &  2000 &   500 &   100 & 2e-08 &    94 &   1.3 & 2e-01 &  1876 &  24.9  \\
12000 &  2400 &   600 &   120 & 3e-08 &   103 &   2.3 & 7e-04 &  1894 &  38.3  \\
14000 &  2800 &   700 &   140 & 2e-08 &    97 &   2.8 & 2e-01 &  1799 &  49.0  \\
16000 &  3200 &   800 &   160 & 3e-08 &   104 &   4.0 & 5e-07 &  1866 &  65.4  \\
18000 &  3600 &   900 &   180 & 4e-08 &   107 &   5.4 & 5e-07 &  1896 &  85.7  \\
20000 &  4000 &  1000 &   200 & 3e-08 &   100 &   5.9 & 6e-04 &  1937 & 106.1  \\
22000 &  4400 &  1100 &   220 & 4e-08 &    93 &   6.3 & 5e-07 &  1847 & 122.4  \\
24000 &  4800 &  1200 &   240 & 2e-08 &    93 &   7.5 & 6e-07 &  1845 & 145.0  \\
26000 &  5200 &  1300 &   260 & 4e-08 &    92 &   8.8 & 6e-07 &  1863 & 172.5  \\
28000 &  5600 &  1400 &   280 & 5e-08 &    96 &  10.7 & 6e-07 &  1941 & 208.7  \\
30000 &  6000 &  1500 &   300 & 4e-08 &    91 &  11.5 & 6e-07 &  1828 & 223.2  \\
\hline
\end{tabular}
\label{table 6.1}
\end{table}

\section{Conclusion and future work}
In this paper, we considered the split feasibility problem, which is to find an element in the intersection of a closed set $ C $ and the linear preimage of another closed set $ D $. We reformulated this problem as an optimization problem with a DC objective, and applied the nonmonotone proximal gradient algorithm with majorization in  \cite[Appendix~A]{LiuPongTake2017} for solving it. We established global convergence and studied local convergence rate of the sequence generated by our algorithm, under mild assumptions. Our numerical experiments demonstrate that our algorithm performs well on solving split feasibility problems that arise from completely positive matrix factorization, sparse matrix factorization and outlier detection.

There are several avenues for future research. For instance, as suggested by one of the referees, a possible future research direction is to extend our approach in Section~\ref{sec7} to find sparse matrix factorization for {\em rectangular} matrices. In this case, the matrix to be factorized is not necessarily symmetric, and hence, one cannot apply \cite[Lemma~2.5]{GroeDur2018} to reformulate this factorization problem into a split feasibility problem involving the set of orthogonal matrices, as in Section~\ref{sec7}. However, note that one can prove the following analogue of \cite[Lemma~2.5]{GroeDur2018}:
\begin{lemma}
		If $X$, $\tilde X\in \R^{m\times r}$ and $Y$, $\tilde Y\in \R^{r\times n}$ are matrices of rank $ r $ and satisfy $XY = \tilde X\tilde Y$, then there exists an invertible matrix $Q\in \R^{r\times r}$ so that
		\[
		X = \tilde X Q \ \ {\rm and}\ \ Y = Q^{-1}\tilde Y.
		\]
\end{lemma}
Thus, if a matrix $ G\in \mathbb{ R}^{m\times n}$ can be factorized as the product of two rank $ r $ matrices $A\in \R^{m\times r}$ and $B\in \R^{r\times n}$, other possible factorizations of $G$ into rank $ r $ matrices of the same sizes can be obtained by multiplying $A$ from the right by an invertible matrix $Q$, and multiplying $B$ from the left by $Q^{-1}$. Now, given an initial factorization $AB$ of $G$ with $A\in \R^{m\times r}$ and $B\in \R^{r\times n}$ having rank $ r $, we can reformulate the sparse factorization problem on $G$ as the following split feasibility problem:
\begin{equation}\label{eq:ssf}
\text{Find }  Q, P\in\R^{r\times r} \text{ s.t. }  (AQ, PB)\in D:=D_s^m\times D^n_s \text{ and } (Q,P)\in C,
\end{equation}
where
$C:=\left\{ (Q,P):\; QP = I\right\}$,
$$ D_s^m:=\left\{U\in\R^{m\times r} :\; \|u_i\| _0\leq s \text{ for each } i=1,\cdots,r
\right\}, $$
$$ D_s^n:=\left\{V\in\R^{r\times n} :\; \|v_j\| _0\leq s \text{ for each } j=1,\cdots,r
\right\}, $$
with $ u_i $, $ v_j $ being the $ i $th column and the $ j $th row of $ U$ and $V $, respectively. One difficulty in using this formulation is that the projection onto the set $C$ may not be easy to compute. Finding efficient ways to project onto $C$ is an interesting future research question.


\appendix

\section{Proof of Theorem~\ref{thm 4.1}}

\begin{proof}
  The boundedness of $\{x^t\}$ follows from Theorem~\ref{thm 3.1}(i). We now prove convergence of the whole sequence. By assumption, $x^*$ is an accumulation point of $\{x^t\}$ so that the function
  \[
  \kappa(u):= \frac{1}{2}d^2(u,D)
  \]
  is continuously differentiable at $Ax^*$ with locally Lipschitz gradient. Then we have from \cite[Example~8.53]{RockWets1998}, \cite[Theorem~1.110(ii)]{Mordukhovich2006} and the chain rule that
  \[
  \nabla(\kappa\circ A)(x^*) = A^T(Ax^* - {\rm Proj}_D(Ax^*)).
  \]
  Using this and the fact that $x^*$ is a stationary point of the split feasibility problem \eqref{eq 1.1} (see Theorem~\ref{thm 3.1}(ii)), we deduce further that
  \[
  \begin{split}
  0 &\in A^T(Ax^* - {\rm Proj}_D(Ax^*)) + N_C(x^*) \\
  & = \nabla(\kappa\circ A)(x^*) + N_C(x^*) = \partial F(x^*),
  \end{split}
  \]
  where the last equality follows from \cite[Exercise~8.8(c)]{RockWets1998}. In particular, it holds that $x^*\in {\rm dom}\partial F$.

  Since $F$ is a KL function and $x^*\in {\rm dom}\partial F$, there exist $\epsilon > 0$ and a continuous concave function $\psi$ as in Definition~\ref{def:KL} so that
  \begin{equation}\label{ineq:KL}
    \psi'(F(x) - F(x^*))\cdot d(0,\partial F(x))\ge 1
  \end{equation}
  whenever $\|x - x^*\|\le \epsilon$ and $F(x^*) < F(x) < F(x^*) + \epsilon$. Moreover, by shrinking $\epsilon$ if necessary, we may assume without loss of generality that $\nabla \kappa$ is globally Lipschitz in $\{Ax:\; x\in B(x^*,\epsilon)\}$ with Lipschitz modulus $\tau$.

  Next, observe from \eqref{eq 3.6} with $M = 0$ that $\{F(x^t)\}$ is nonincreasing. Since $F$ is also nonnegative, we deduce that the limit $\lim\limits_{t\to\infty}F(x^t)$ exists. In addition, notice that $F$ is continuous in its closed domain and $x^*$ is an accumulation point of $\{x^t\}$. Thus, we conclude that $\lim\limits_{t\to\infty}F(x^t) = F(x^*)$.

  Now, if $F(x^{t_0}) = F(x^*)$ for some $t_0 \ge 0$, then we see from \eqref{eq 3.6} with $M = 0$ and $\lim\limits_{t\to\infty}F(x^t) = F(x^*)$ that $x^{t+1}=x^t$ for all $t\ge t_0$, which implies that the sequence $\{x^t\}$ converges (finitely). Thus, from now on, we focus on the case that $F(x^t) > F(x^*)$ for all $t\ge 0$.

  In this case, note from Lemma~\ref{lemma 3.2} that there exists $N_0 > 1$ so that $\|x^{t}-x^{t-1}\|\le \frac{\epsilon}{2}$ whenever $t\ge N_0$. Also, using Lemma~\ref{lemma 3.2}, the definition of accumulation point and the fact that $\lim\limits_{t\to\infty}F(x^t) = F(x^*)$, there exists $N_1 \ge N_0$ so that
  \begin{enumerate}[{\rm (i)}]
  \item $\|x^{N_1}-x^*\| \le \frac{\epsilon}2$ and $F(x^*) < F(x^{N_1}) < F(x^*) + \epsilon$.
  \item $\|x^{N_1}-x^*\| + \|x^{N_1} - x^{N_1-1}\| + C_1\psi(F(x^{N_1})-F(x^*))\le \frac{\epsilon}{2}$,
  \end{enumerate}
  where $C_1 := \frac{2(\tau \lambda_{\max}(A^TA) + \beta)}{c}$, $c$ is as in \eqref{eq 3.6}, $\tau$ is the Lipschitz continuity modulus of $\nabla \kappa$ on $\{Ax:\; x\in B(x^*,\epsilon)\}$, $\beta = \sup_t \bar L_t$ with $\bar L_t$ defined in {\bf Step 2} of ${\sf SpFeas}_{{\rm DC}_{\sf ls}}$, and $\beta$ is finite according to Lemma~\ref{lemma 3.2}.

  We claim that if $t\ge N_1$ and $\|x^t - x^*\|\le \epsilon/2$, then
  \begin{equation}\label{key_rel}
    2\|x^{t+1}-x^t\|\le \|x^t-x^{t-1}\| +  C_1 \left[\psi(F(x^t) - F(x^*)) - \psi(F(x^{t+1}) - F(x^*))\right].
  \end{equation}
  To this end, note that since $x^t\in B(x^*,\epsilon/2)$ and $t\ge N_1\ge N_0$, we have $\|x^{t}-x^{t-1}\|\le \frac{\epsilon}{2}$ and hence $\|x^{t-1} - x^*\|\le \epsilon$. Thus, $\kappa$ is continuously differentiable at $Ax^{t-1}$ and $Ax^t$. Moreover, we see from \cite[Example~8.53]{RockWets1998} and \cite[Theorem~1.110(ii)]{Mordukhovich2006} (see also \eqref{tocite}) that $\nabla (\kappa\circ A)(x^{t-1}) = A^T(Ax^{t-1} - {\rm Proj}_D(Ax^{t-1}))$. Using this and the definition of $x^t$, we deduce that
\[
x^{t}\in {\rm Proj}_C\left(x^{t-1}-\frac{\nabla (\kappa\circ A)(x^{t-1})}{\bar{L}_{t-1}}\right).
\]
Thus, according to \eqref{normalcone},
\[
v^t:={\bar L}_{t-1}(x^{t-1}-x^t)-\nabla (\kappa\circ A)(x^{t-1})\in N_C(x^t).
\]
Moreover, using the definition of $v^t$, we have
\begin{equation}\begin{split}\label{eq 4.4}
\Vert v^t+\nabla (\kappa\circ A)(x^t)\Vert&\leq\Vert \nabla (\kappa\circ A)(x^t)-\nabla (\kappa\circ A)(x^{t-1})\Vert+{\bar L}_{t-1}\Vert x^{t}-x^{t-1}\Vert\\
                       &\leq(\tau \lambda_{\max}(A^TA) + \beta)\Vert x^{t}-x^{t-1}\Vert,
\end{split}\end{equation}
where the second inequality holds because $\beta = \sup_{t}\bar L_t$ and $\nabla \kappa$ is globally Lipschitz in $\{Ax:\;x\in B(x^*,\epsilon)\}$ with Lipschitz modulus $\tau$.
Since $v^t+\nabla (\kappa\circ A)(x^t)\in N_C(x^t) + \nabla (\kappa\circ A)(x^t) = \partial F(x^t)$, we obtain from \eqref{eq 4.4} that
  \[
  d(0,\partial F(x^t))\le \left(\tau \lambda_{\max}(A^TA) +\beta\right)\|x^t-x^{t-1}\|.
  \]
Making use of this, the concavity of $\psi$ and \eqref{eq 3.6} with $M = 0$, we see further that
  \[
  \begin{split}
    &\left(\tau \lambda_{\max}(A^TA) + \beta\right)\|x^t-x^{t-1}\|\cdot \left[\psi(F(x^t) - F(x^*)) - \psi(F(x^{t+1}) - F(x^*))\right]\\
    & \ge d(0,\partial F(x^t))\cdot \left[\psi(F(x^t) - F(x^*)) - \psi(F(x^{t+1}) - F(x^*))\right]\\
    & \ge d(0,\partial F(x^t))\cdot \psi'(F(x^t) - F(x^*))\cdot \left[F(x^t) - F(x^{t+1})\right]\\
    & \ge \frac{c}2\|x^{t+1}-x^t\|^2,
  \end{split}
  \]
  where the last inequality follows from \eqref{eq 3.6} with $M=0$, \eqref{ineq:KL}, and the facts that $\|x^t-x^*\|\le \epsilon/2$ and that $F(x^*) < F(x^t)\le F(x^{N_1}) < F(x^*)+\epsilon$ (since $t\ge N_1$). Dividing both sides of the above inequality by $\frac{c}2$, taking square root, using the relation $\sqrt{ab}\le \frac{a+b}{2}$ for any nonnegative numbers $a$ and $b$ and invoking the definition of $C_1$, we obtain further that
  \[
  \begin{split}
  \|x^{t+1}-x^t\|& \le \sqrt{\|x^t-x^{t-1}\|\cdot C_1\left[\psi(F(x^t) - F(x^*)) - \psi(F(x^{t+1}) - F(x^*))\right]}\\
  & \le \frac12\left(\|x^t-x^{t-1}\| +  C_1\left[\psi(F(x^t) - F(x^*)) - \psi(F(x^{t+1}) - F(x^*))\right]\right),
  \end{split}
  \]
  from which \eqref{key_rel} follows immediately.

  Next, we show by induction that $x^t\in B(x^*,\epsilon/2)$ whenever $t\ge N_1$. The case $t = N_1$ follows from construction. Suppose that $x^t\in B(x^*,\epsilon/2)$ whenever $t = N_1, \ldots, N_1+k-1$ for some $k\ge 1$. Then
  \[
  \begin{split}
    &\|x^{N_1+k} - x^*\|  \le \|x^{N_1}-x^*\| + \sum_{t=N_1}^{N_1+k-1}\|x^{t+1}-x^t\|\\
    & \overset{\rm (a)}\le \|x^{N_1}-x^*\| \\
    &\ \ + \sum_{t=N_1}^{N_1+k-1}\left(\|x^t-x^{t-1}\| - \|x^{t+1}-x^t\| +  C_1 \left[\psi(F(x^t) - F(x^*)) - \psi(F(x^{t+1}) - F(x^*))\right]\right)\\
    & \le \|x^{N_1}-x^*\| + \|x^{N_1}-x^{N_1-1}\| + C_1\psi(F(x^{N_1}) - F(x^*)) \overset{\rm (b)}\le \frac{\epsilon}{2},
  \end{split}
  \]
  where (a) follows from the induction hypothesis and \eqref{key_rel}, and (b) follows from the definition of $N_1$. Thus, $x^t\in B(x^*,\epsilon/2)$ whenever $t\ge N_1$ by induction.

  Since $x^t\in B(x^*,\epsilon/2)$ whenever $t\ge N_1$, we can sum both sides of \eqref{key_rel} from $N_1$ to $\infty$ and obtain
  \[
  \begin{split}
    &\sum_{t=N_1}^\infty\|x^{t+1}-x^t\|\\
    & \le \sum_{t=N_1}^\infty\bigg(\|x^t-x^{t-1}\| - \|x^{t+1}-x^t\| +  C_1 \left[\psi(F(x^t) - F(x^*)) - \psi(F(x^{t+1}) - F(x^*))\right]\bigg)\\
    & \le \|x^{N_1}-x^{N_1-1}\| +C_1\psi(F(x^{N_1}) - F(x^*)) < \infty.
  \end{split}
  \]
  Thus, the sequence $\{x^t\}$ is Cauchy and is hence convergent.
\qed\end{proof}



\end{document}